\documentclass{article}

\usepackage[english]{babel}
\usepackage[UTF8]{ctex}
\usepackage[letterpaper,top=2cm,bottom=2cm,left=3cm,right=3cm,marginparwidth=1.75cm]{geometry}
\usepackage{tikz}
\usetikzlibrary{intersections}
\usetikzlibrary{arrows.meta,decorations,calligraphy}
\usepackage{amsmath}
\usepackage{graphicx}

\usepackage[colorlinks=true, allcolors=blue]{hyperref}
\usepackage{graphicx}
\usepackage{amsfonts}
\usepackage{amstext}

\usepackage{amsthm}
\usepackage{amssymb,amsmath,amsbsy}
\usepackage{verbatim}
\usepackage{thm-restate}
\newcommand{\p}{\mathcal{P}}
\newcommand{\dff}{\draw[fill,color=gray!10]}
\newcommand{\df}{\draw[fill]}
\newcommand{\dfr}{\draw[color=red,thick]}
\newcommand{\dfb}{\draw[fill,color=blue!10]}
\newcommand{\bt}{\begin{tikzpicture}}
\newcommand{\et}{\end{tikzpicture}}
\providecommand{\keywords}[1]
{
  \small	
  \textbf{\textit{Keywords---}} #1
}
\newtheorem{theorem}{Theorem}
\newtheorem{conjecture}[theorem]{Conjecture}
\newtheorem{corollary}[theorem]{Corollary}
\newtheorem{lemma}[theorem]{Lemma}
\newtheorem{definition}[theorem]{Definition}

\tikzset{dot/.style={font=\LARGE}}
\title{The planar Tur\'an number of the seven-cycle}
\author{Ervin Gy\H{o}ri\thanks
  {R\'enyi Institute, Budapest,
 Hungary.  Research partially supported by the NKFIH Grant 132696. E-mail: {\tt gyori.ervin@renyi.hu}}
  \\ Alan Li\thanks{ Amherst College, 220 South Pleasant Street,
Amherst, MA 01002  E-mail:
 {\tt ali24@amherst.edu}}
 \qquad
Runtian Zhou \thanks{
Davidson College, Davidson, North Carolina 28035
Email: {\tt dazhou@davidson.edu}}
}

\begin{document}
\maketitle

\begin{abstract}
The planar Tur\'an number, $ex_\p(n,H)$, is the maximum number of edges in an $n$-vertex planar graph which does not contain $H$ as a subgraph. The topic of extremal planar graphs was initiated by Dowden (2016). He obtained sharp upper bound for both $ex_\p(n,C_4)$ and $ex_\p(n,C_5)$. Later on, D. Ghosh et al. obtained sharp upper bound of $ex_\p(n,C_6)$ and proposed a conjecture on $ex_\p(n,C_k)$ for $7 \leq k\leq 10$. In this paper, we give a sharp upper bound $ex_\p(n,C_7)\leq {18\over 7}n-{48\over 7}$, which satisfies the conjecture of D. Ghosh et al. It turns out that this upper bound is also sharp for $ex_\p(n,\{K_4,C_7\})$, the maximum number of edges in an $n$-vertex planar graph which does not contain $K_4$ or $C_7$ as a subgraph.
\end{abstract}
\keywords{Planar Tur\'an number, Extremal planar graph}
\section{Introduction and Main Results}
In this paper, all graphs considered are planar, undirected, finite and contain neither loops nor multiple edges (unless otherwise stated). More specifically, we study plane graphs which are embeddings (drawings) of graphs in the plane such that the edge curves do not cross each other; they may share just endpoints. We use $C_k$ to denote the cycle of $k$ vertic es and $K_r$ to denote the complete graph of $r$ vertices. We use $n$-face to denote a face with $n$ edges. In particular, we use $(8+)$-face to denote a face with at least $8$ edges.

The Tur\'an number $ex(n,H)$ for a graph $H$ is the maximum number of edges in an $n$-vertex graph with no copy of $H$ as a subgraph. The first result on the topic of Tur\'an numbers was obtained by Tur\'an, who proved that the balanced complete $r$-partite graph is the unique extremal graph of $ex(n,K_{r+1})$ edges. The Erd\H{o}s-Stone-Simonovits theorem \cite{erdos1963structure,erdos1962number} 
then generalizes this result and asymptotically determines $ex(n,H)$ for all nonbipartite graphs $H:$ $ex(n,H)=(1-{1\over \mathcal{X}(H)-1}){n\choose 2}+o(n^2)$.

In 2016, Dowden et al. \cite{dowden2016extremal} initiated the study of Tur\'an-type problems when host graphs are plane graphs, i.e., how many edges can a plane graph on $n$ vertices have, without containing a given graph as a subgraph? The planar Tur\'an number of a graph $H$, $ex_{\mathcal{P}}(n,H)$, is the maximum number of edges in a planar graph on $n$ vertices which does not contain $H$ as a subgraph.  Dowden et al. \cite{dowden2016extremal} obtained the tight bounds  $ex_{\mathcal{P}}(n,C_4) \leq\frac{15(n-2)}{7}$, for all $n\geq 4$ and $ex_{\mathcal{P}}(n,C_5) \leq\frac{12n-33}{5}$, for all $n\geq 11$.  For $k\in \{4,5\}$, let $\Theta_k$ denote the graph obtained from $C_k$ by adding a chord. Y. Lan et al. \cite{lan2019extremal} showed that $ex_\p(n,\Theta_4)\leq {15(n-2)\over 5}$ for all $n\geq 4$, $ex_\p(n,\Theta_5)\leq {5(n-2)\over 2}$ for all $n\geq 5$ and $ex_\p(n,C_6)\leq {18(n-2)\over 7}$. The bounds for $ex_\p(n,\Theta_4)$ and $ex_\p(n,\Theta_5)$ are tight infinitely often, but the upper bound for $ex_\p(n,C_6)$ was improved by D. Ghosh et al. \cite{ghosh2022planar}. They proved $ex_\p(n,C_6)\leq {5n-14\over 2}$ for all $n\geq 18$. In the same paper, they conjectured an upper bound for $ex_\p(n,C_k)$ for each $7\leq k\leq 10$ and large $n$.
\begin{conjecture}
(\cite{ghosh2022planar}) For each $7 \leq k\leq 10$ and sufficiently large $n$, $$ex_\p(n,C_k)\leq 3n-6-{3n+6\over k}.$$
\end{conjecture}

Recently, Conjecture 1 was shown to be false for $k\geq 11$ by Cranston et al. \cite{cranston2021planar}, and independently by Lan et al. \cite{lan2022improved} In this paper, we focus on $k=7$ and prove the conjecture in this case. Moreover, we show this upper bound is sharp for infinitely many $n$. We also find a construction which shows the bound is also sharp for $ex_\p(n,\{K_4,C_7\})$.

We denote the vertex and the edge sets of a graph $G$ by $V(G)$ and $E(G)$ respectively. We also denote the number of vertices and edges of $G$ by $v(G)$ and $e(G)$ respectively. The minimum degree of $G$ is denoted $\delta(G)$. The main ingredient of the result is as follows:
\begin{theorem}\label{pt}
Let $G$ be a $2$-connected, $C_7$-free plane graph on $n$ ($n\geq 8$) vertices with $\delta(G)\geq 3$ and with any adjacent vertices having total degree at least $7$. Then $$e(G)\leq \frac{18}{7}n - \frac{48}{7}.$$
\end{theorem}
Here is a brief summary about how we will prove this theorem. First we apply Lemma \ref{replacement} to obtain another plane graph $G'$ with some properties and it suffices to prove $G'$ satisfies the inequality of Theorem \ref{pt}. Then, with the trick of transformation into dual graph, we use Lemma \ref{smallblockslemma} and Lemma \ref{smallblocksonly} to obtain $e(G')\leq \frac{18}{7}n - \frac{48}{7}.$

We then use Theorem \ref{pt} in order to establish our desired result, which gives the upper bound of $\frac{18}{7}n - \frac{48}{7}$ for all $C_7$-free plane graphs with at least $60$ vertices.
\begin{theorem}\label{mt}
Let $G$ be a $C_7$-free plane graph on $n$ ($n\geq 60$) vertices. Then $$e(G)\leq \frac{18}{7}n - \frac{48}{7}.$$
\end{theorem}
\begin{corollary}\label{gt}
Let $G$ be a $C_7$-free plane graph on $n$ vertices such that $G$ contains no maximal $2$-connected subgraph of more than $7$ vertices. Then $$e(G)\leq {5\over 2}n - \frac{18}{7}.$$
\end{corollary}
This corollary follows from the proof of Theorem \ref{mt}.
\begin{corollary}\label{kc}
Let $G$ be a $K_4,C_7$-free plane graph on $n$ ($n\geq 60$) vertices. Then $$e(G)\leq \frac{18}{7}n - \frac{48}{7}.$$
\end{corollary}
We show that, for infinitely many large graphs, Corollary \ref{kc}, and therefore Theorem \ref{mt}, is tight:
\begin{theorem}\label{ct}
For every $n\geq 110, n\cong 26\pmod{42}$, there exists a $K_4,C_7$-free plane graph $G$ with $e(G)=\frac{18}{7}n - \frac{48}{7}.$
\end{theorem}
\begin{corollary}
For every $n\geq 110,n\cong 26\pmod{42}$, there exists a $C_7$-free plane graph $G$ with $e(G)=\frac{18}{7}n - \frac{48}{7}.$
\end{corollary}

\section{Proof of Theorem \ref{ct}: Extremal Graph Construction}
This construction is a modification of the construction of D. Ghosh et al. \cite{ghosh2022planar}. First we show that for each $k\geq 2$, there is a plane graph $G_0$ which contains $2$ rows by $k$ columns of ``hexagons" and an additional vertex, where each face of $G_0$ has length $8$ and each vertex of $G_0$ has degree $2$ or $3$ (Figure \ref{f1}). Then, from $G_0$ we can construct $G$, where $G$ is a $K_4,C_7$-free plane graph with $v(G)=42k+26$ and $e(G)=108k+60$, so $e(G)=\frac{18}{7}v(G) - \frac{48}{7}$.

Given $G_0$, we make local changes to each vertex of $G_0$, as shown in Figure \ref{cs}. Specifically, if $v\in v(G_0)$ has degree $3$, then $v_1,v_2,v_3$ are the middle points of the three edges incident to $v$ in $G_0$, and $v_4,v_5,v_6$ are three vertices added in replacement of $v$. If $v\in v(G_0)$ has degree $2$, then $v_1,v_3$ are the middle points of the two edges incident to $v$ in $G_0$, and $v_2,v_4,v_5,v_6$ are four vertices added in replacement of $v$. Hence each edge of $G_0$ become a vertex in $G$ (the middle point of that edge), each degree $3$ vertex of $G_0$ is replaced with $3$ vertices, and each degree $2$ vertex of $G_0$ is replaced with $4$ vertices. Also, as shown in Figure \ref{cs}, each vertex of $G_0$ contributes $12$ edges in $G$. Observe that in $G_0$, there are $3k+7$ vertices with degree $2$ and $6k-2$ vertices with degree $3$, and $e(G_0)=12k+4$. Then, it is easy to verify $G$ is a $K_4,C_7$-free graph with
\begin{align*}
v(G)&=e(G_0)+3(6k-2)+4(3k+7)=42k+26,\\
e(G)&=12v(G_0)=12((6k-2)+(3k+7))=108k+60,\\
\Rightarrow e(G)&={18\over 7}v(G)-{48\over 7}.
\end{align*}

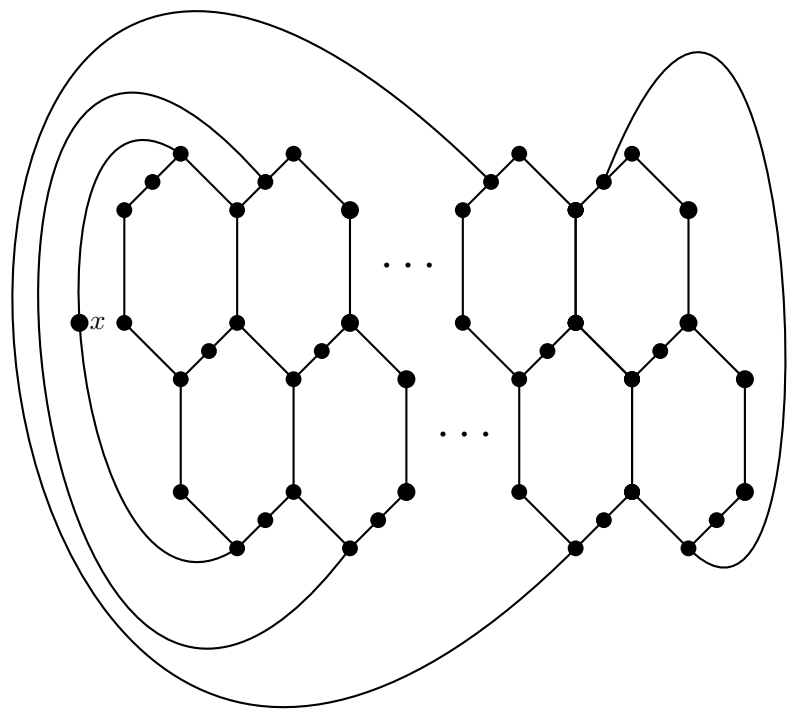
\begin{figure}

\begin{tikzpicture}[style=thick,scale=1.5]
\draw (2,1)--(2,0)--(2.5,-0.5)--(2.5,-1.5);
\foreach \x in {-1.5,-0.5,1.5,2.5}
{
	\draw (\x,1.5)--  ++(-0.5,-0.5)-- ++(0,-1)-- ++(0.5,-0.5)-- ++(0,-1)-- ++(0.5,-0.5)-- ++(0.5,0.5);
	\draw(\x,1.5)--(\x+0.5,1);
	\draw(\x+0.5,0)--(\x,-0.5);
	\fill (\x-0.25,1.25)circle(2pt);
	\fill (\x+0.25,-0.25)circle(2pt);
	\fill (\x+0.75,-1.75)circle(2pt);
    \fill (\x,-1.5)circle(2pt);
    \fill (\x,1.5)circle(2pt);
    \fill (\x,-0.5)circle(2pt);
    \fill (\x-0.5,1)circle(2pt);
    \fill (\x-0.5,0)circle(2pt);
    \fill (\x+0.5,-2)circle(2pt);
    \fill (2,1)circle(2pt);
    \fill (2,0)circle(2pt);
    \fill (2.5,-0.5)circle(2pt);
    \fill (2.5,-1.5)circle(2pt);
}

\node[dot] at (0.55,0.5) {$\cdots$};
\node[dot] at (1.05,-1) {$\cdots$};
\draw[fill] (0,1)circle(2pt)--(0,0)circle(2pt)--(0.5,-0.5)circle(2pt)--(0.5,-1.5)circle(2pt);
\draw[fill] (3,1)circle(2pt)--(3,0)circle(2pt)--(3.5,-0.5)circle(2pt)--(3.5,-1.5)circle(2pt);
\path [draw,name path=add line][black] (-1.5,1.5)..controls(-3,2.5)and(-2.5,-3)..(-1,-2);
\path [name path=ver line](-1.75,0)--(-4,0);
\draw [name intersections={of=add line and ver line, by=x}][thick,fill=black]  (x) circle(2pt);
\node[right] at (x){$x$};
\draw (-0.75,1.25)..controls(-4,5)and(-3,-6)..(0,-2);
\draw (2.25,1.25)..controls(3.9,5.5)and(4.5,-3.5)..(3,-2);
\draw (1.25,1.25)..controls(-5,7.5)and(-4,-8)..(2,-2);

\end{tikzpicture}
\caption{\label{f1}Graph $G_1$}
\end{figure}
\begin{figure}
\centering
\begin{tikzpicture}
\draw[fill,xshift=-4cm](90:0)node[anchor=-30]{$v$}circle(2pt)--(90:1.8)circle(2pt);
\draw[fill,xshift=-4cm](90:0)--(-30:1.8)circle(2pt);
\draw[fill,xshift=-4cm](90:0)--(-150:1.8)circle(2pt);
\draw[fill,xshift=4cm,dashed](90:1.2)circle(2pt)node[right]{$v_1$}--(90:1.8)circle(2pt);
\draw[fill,xshift=4cm,dashed](-30:1.2)circle(2pt)node[below]{$v_2$}--(-30:1.8)circle(2pt);
\draw[fill,xshift=4cm,dashed](-150:1.2)circle(2pt)node[below]{$v_3$}--(-150:1.8)circle(2pt);
\draw[xshift=4cm](90:1.2)--(-30:1.2)--(-150:1.2)--cycle;
\draw[fill,xshift=4cm](90:1.2)--(30:0.3)node[anchor=-150]{$v_6$}circle(2pt);
\draw[fill,xshift=4cm](90:1.2)--(150:0.3)node[anchor=-30]{$v_5$}circle(2pt);
\draw[fill,xshift=4cm](-30:1.2)--(30:0.3)circle(2pt);
\draw[fill,xshift=4cm](-30:1.2)--(-90:0.3)circle(2pt);
\draw[fill,xshift=4cm](-150:1.2)--(-90:0.3)node[anchor=90]{$v_4$}circle(2pt);
\draw[fill,xshift=4cm](-150:1.2)--(150:0.3)circle(2pt);
\draw[xshift=4cm](-90:0.3)--(150:0.3)--(30:0.3)--cycle;
\draw[ultra thick,dotted,->](-1cm,0.5cm)--(1cm,0.5cm);
\end{tikzpicture}
\begin{tikzpicture}
\draw[fill,xshift=-4cm](90:0)node[anchor=-30]{$v$}circle(2pt)--(90:1.8)circle(2pt);

\draw[fill,xshift=-4cm](90:0)--(-150:1.8)circle(2pt);
\draw[fill,xshift=4cm,dashed](90:1.2)circle(2pt)node[right]{$v_1$}--(90:1.8)circle(2pt);
\draw[fill,xshift=4cm](-30:1.2)circle(2pt)node[below]{$v_2$};
\draw[fill,xshift=4cm,dashed](-150:1.2)circle(2pt)node[below]{$v_3$}--(-150:1.8)circle(2pt);
\draw[xshift=4cm](90:1.2)--(-30:1.2)--(-150:1.2)--cycle;
\draw[fill,xshift=4cm](90:1.2)--(30:0.3)node[anchor=-150]{$v_6$}circle(2pt);
\draw[fill,xshift=4cm](90:1.2)--(150:0.3)node[anchor=-30]{$v_5$}circle(2pt);
\draw[fill,xshift=4cm](-30:1.2)--(30:0.3)circle(2pt);
\draw[fill,xshift=4cm](-30:1.2)--(-90:0.3)circle(2pt);
\draw[fill,xshift=4cm](-150:1.2)--(-90:0.3)node[anchor=90]{$v_4$}circle(2pt);
\draw[fill,xshift=4cm](-150:1.2)--(150:0.3)circle(2pt);
\draw[xshift=4cm](-90:0.3)--(150:0.3)--(30:0.3)--cycle;
\draw[ultra thick,dotted,->](-1cm,0.5cm)--(1cm,0.5cm);
\end{tikzpicture}
\caption{\label{cs}Local change from $G_0$ to $G$}
\end{figure}
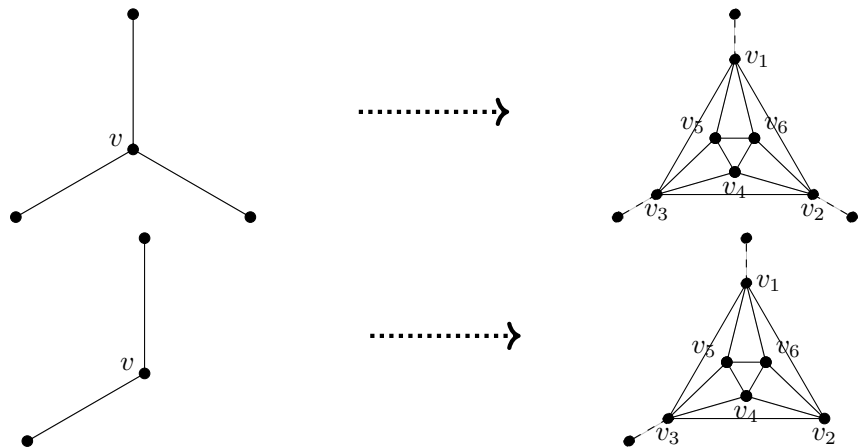
\section{Definitions and Preliminaries}

\begin{definition}\label{tb}
Let $G$ be a plane graph. Then a \textbf{triangular block} $(E, F)$ consists of a set of edges $E \in E(G)$ and a set of faces $F \in F(G)$ built as follows: \\
1) Begin with an edge $e \in E(G)$. If it is not in any $3$-face of $G$, then we have $E = \{e\}$ and $F = \emptyset$; \\
2) Otherwise, we add $e$ into $E$, and for each edge $e' \in E$, we add all $3$-faces containing $e'$ into $F$, and all edges in such faces into $E$; \\
3) Repeat step 2) until no more faces can be added into $F$.

\end{definition}
Note that unlike the definition in \cite{ghosh2022planar}, we account for faces in our triangular blocks as well. This is because there are many situations where two different triangular blocks can have the same edge set but different face sets, and this makes a big difference in our calculations. \\

Similar to \cite{ghosh2022planar}, we make the following observations: \\
(i) Given a triangular block $B$, no matter which edge we begin with, we always obtain $B$ as the triangular block. \\
(ii) The triangular blocks form a partition of the edges in $G$.
\begin{definition}
Let $G$ be a $2$-connected plane graph and $e\in E(G)$. Let the two faces incident to $e$ have length $l_1$ and $l_2$. The \textbf{contribution of $\mathbf{e}$} to the face number of $G$, denoted by $f^*(e)$, is defined as $$f^*(e)=\max ({1\over l_1},{1\over 8})+\max ({1\over l_2},{1\over 8}).$$
\end{definition}

\begin{definition}
Let $G$ be a plane graph, and let $B=(E,F)$ be a triangular block of $G$. An edge in $E$ is called an \textbf{exterior edge} of $B$ if it borders a non-triangular face of $G$. All other edges of $E$ are called \textbf{interior edges}. The set of all exterior edges of $E$ is called the \textbf{boundary} of $B$. A non-triangular face in $G$ that borders an exterior edge is called an \textbf{exterior face} of $B$.
Now if $G$ is a $2$-connected plane graph, we denote
\begin{align*}
e_B^\partial&=\text{number of edges of $B$ adjacent to some $(8+)$-face}\\
e_B^{int}&=\text{number of edges of $B$ that are not adjacent to any $(8+)$-face}\\
f_B&=\sum_{e\in E}f^*(e).
\end{align*}

\end{definition}

\begin{definition}\label{ltb}
\textbf{Large triangular blocks} are triangular blocks with a boundary that forms a single cycle, and such that, if the boundary does not form an outerface, then it can only border $(8+)$-faces, otherwise a $C_7$ will form. Triangular blocks that can border faces of length less than $7$, or have boundary that does not form a single cycle, are called \textbf{small triangular blocks}. Given a plane graph $G$, denote $S(G)$ as the set of its small triangular blocks.
\end{definition}

Note that in the definition above we mention the boundary because there can be many triangular blocks with boundaries that do not form a single cycle.

\begin{definition}\label{sbs}
Let $G$ be a plane graph. Then a \textbf{small block set} is a set of small triangular blocks given by the following algorithm: \\
1) Start with any small triangular block as the set; \\
2) For each exterior face of triangular blocks in the set, if the face is of length less than $7$, then we add all triangular blocks adjacent to that face into the set; \\
3) Repeat step 2) until the set cannot be further enlarged.

Similar to triangular blocks, small block sets have the following property: \\
(i) Given a small block set $S$, no matter which small triangular block we begin with, we always obtain $S$ as the small block set. \\
(ii) The small block sets and large triangular blocks form a partition of the edges in $G$.
\end{definition}

\begin{definition}
Denote $T_3$ as the triangular block shown in Figure \ref{t3}. For any figures of triangular blocks in this paper, we use gray to refer to faces that belong to the triangular block.

\begin{figure}
\centering
\begin{tikzpicture}
\dff (90:2)--(-30:2)--(-150:2)--cycle;
\draw[fill](90:2)circle(2pt)node[right]{$v_1$};
\draw[fill](-30:2)circle(2pt)node[below]{$v_2$};
\draw[fill](-150:2)circle(2pt)node[below]{$v_3$};
\draw (90:2)--(-30:2)--(-150:2)--cycle;
\draw[fill](90:2)--(30:0.5)node[anchor=-150]{$v_6$}circle(2pt);
\draw[fill](90:2)--(150:0.5)node[anchor=-30]{$v_5$}circle(2pt);
\draw[fill](-30:2)--(30:0.5)circle(2pt);
\draw[fill](-30:2)--(-90:0.5)circle(2pt);
\draw[fill](-150:2)--(-90:0.5)node[anchor=90]{$v_4$}circle(2pt);
\draw[fill](-150:2)--(150:0.5)circle(2pt);
\draw (-90:0.5)--(150:0.5)--(30:0.5)--cycle;
\end{tikzpicture}
\caption{\label{t3}Triangular block $T_3$}
\end{figure}
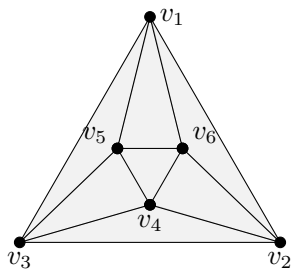
\end{definition}
\section{Lemmas for proving Theorem \ref{pt}}

\begin{lemma}\label{8+boundary}
Let $G$ be  a $C_7$-free, $2$-connected plane graph and let $F$ be a face of $G$ with length $l$ and $l\geq 8$. Then $F$ is incident to at most $l-8$ number of $T_3$ that share two edges with $F$.
\end{lemma}
\begin{proof}
Let the cycle formed by edges of $F$ be $C$ and let $a, b, c$ be the three outer vertices of a $T_3$, such that the $T_3$ shares edges $(ab),(bc)$ with $F$. Now observe that if $(ab),(bc)$ are not consecutive on $C$, then we will have Figure \ref{3path}. The gray area is $F$, and the two blue curves are trails on the boundary of $F$ each with at least $1$ vertex other than endpoints (not presented in Figure \ref{3path}). Now deleting $b$ will leave the vertices on the two different blue path disconnected, so $b$ is a cut vertex, therefore $G$ is not $2$-connected. Thus, we must have that $(ab),(bc)$ are consecutive on $C$.

Since $(ab),(bc)$ are consecutive on $C$, we know $C-\{(ab),(bc)\}+\{(ac)\}$ is a cycle of length $l-1$. If there is another $T_3$ incident to $F$ that shares two edges, we can repeat the same process (deleting the two shared edges from the cycle and adding the other edge on the boundary of $T_3$ to reduce the length of the cycle by $1$). We can always do this because no $T_3$'s can be incident to each other in a $C_7$-free graph. Since $G$ is $C_7$-free, it follows that $F$ is incident to at most $l-8$ number of $T_3$'s that share two edges with $F$.
\begin{figure}
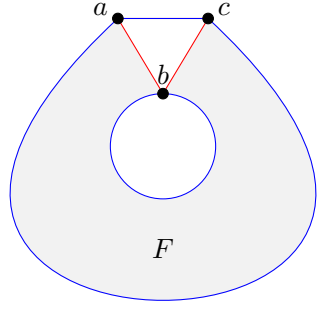

\centering
\bt
\dff (-0.6,2)..controls(-6,-3) and (6,-3)..(0.6,2)--cycle;
\draw[fill,color=white] (0,0.3)circle(0.7);
\draw[fill,color=white] (0,1)--(-0.6,2)--(0.6,2)--cycle;
\draw[color=red](0,1)--(-0.6,2)--(0.6,2)--cycle;
\draw[color=blue](-0.6,2)..controls(-6,-3) and (6,-3)..(0.6,2)--cycle;
\draw[color=blue](0,0.3)circle(0.7);
\df (-0.6,2)circle(2pt);
\df (0.6,2)circle(2pt);
\df(0,1)circle(2pt);
\draw (-0.6,2)node [anchor=-30]{$a$};
\draw (0.6,2)node [anchor=-150]{$c$};
\draw (0,1)node [anchor=-90]{$b$};
\draw (0,-0.8)node[anchor=90]{$F$};

\et
\caption{\label{3path}If 
 $(ab), (bc)$ are non-consecutive on $C$, then $G$ cannot be $2$-connected}
\end{figure}

\end{proof}
\begin{lemma}\label{replacement}

Let $G'$ be a $2$-connected, $C_7$-free plane graph on $n$ $(n\geq 8)$ vertices with $\delta(G')\geq 3$ and with any two adjacent vertices having total degree at least $7$. Then, from $G'$, we can construct another plane graph $G$ with the following properties:

$(1)$ If $e(G')>{18\over 7}v(G')-{48\over 7}$, then $e(G)>{18\over 7}v(G)-{48\over 7}.$

$(2)$ $v(G)\geq v(G')$;

$(3)$ $G$ is $2$-connected with $\delta(G)\geq 3$ and with any adjacent vertices having total degree at least $7$;

$(4)$ Since $G$ is a plane graph, we can consider the ``interior" of any cycle of $G$. We claim that each $7$-cycle of $G$ contains an $(8+)$-face in its interior; moreover, each $7$-cycle contains at least one edge that belongs to some $T_3$;

$(5)$ $G$ contains no small block set with only one edge;

$(6)$ $G$ contains no small block set with less than $7$ vertices while its boundary forms a single cycle;

$(7)$ $G$ contains no large triangular blocks other than $T_3$;

$(8)$ $G$ contains only triangular blocks that could appear in a $C_7$-free plane graph;

$(9)$ If $F$ is an $(8+)$-face of $G$ and it's adjacent to two edges of a $T_3$, then the cycle formed by it can the shortened by the processed described in Lemma \ref{8+boundary}. We claim that in this way, this cycle cannot be shortened into a $C_7$.
\end{lemma}

In this lemma we are trying to create a simplified graph where we replace things like large triangular blocks with $T_3$'s, such that by property $(1)$ and $(2)$ it suffices to prove the desired bound ($e \leq {18\over 7}v - {48\over 7}$) for this new graph. In this simplified graph, property $(3),(4),(5),(6),(7),$ and $(8)$ allow us to prove Lemma \ref{smallblockslemma} by calculating vertex and face contribution for each of the remaining small triangular blocks. Then, with property $(9)$ and Lemma \ref{smallblockslemma}, we manage to prove this bound ($e \leq {18\over 7}v - {48\over 7}$) by using dual graphs, as outlined in the proof of Theorem \ref{pt}.
\begin{proof}
Since $G'$ is $2$-connected, every edge borders two faces. We will construct $G$ inductively by using $T_3$ to replace those items that we don't want. Specifically, let $T$ be one of the following in $G'$: an edge that borders two $(8+)$-faces, a small block set on less than $7$ vertices whose boundary forms a single cycle, or a large triangular block. We will replace $T$ with $T_3$'s and continue doing replacement in the rest of the graph. Since each time we reduce by $1$ the number of items we don't want, finally we'll get a graph $G$ satisfying property $(5),(6),$ and $(7)$. As we drew all possible small triangular blocks in a $C_7$-free plane graph, we found none of them has more than $6$ vertices. Thus we remove at most $6$ vertices each time. Since we add at least a $T_3$ each time, which contains at least $6$ vertices, property $(2)$ is satisfied. In our construction process, it will be clear that properties $(3)$ and $(4)$ are true. Property $(8)$ is true because the edges we added are only contained in $T_3$ triangular blocks, and the original triangular blocks of $G$, if changed during the replacement process, are either deleted or become a plane subgraph of original triangular block.

Property $(9)$ is a bit tricky. Notice that we did not create any new $(8+)$-faces in the process described below and that during each replacement, the shortest distance between any two vertices is not shortened. Therefore, if the  cycle formed by an $(8+)$-face is shortened only by the new $T_3$'s we added, it can at most be shortened to the size of the original $(8+)$-face in $G$. If this can be further shortened to a $C_7$ using original(not added) $T_3$'s in $G'$, then we have a $C_7$ in $G'$ which is not allowed.

We will prove property $(1)$ as we describe the construction process.

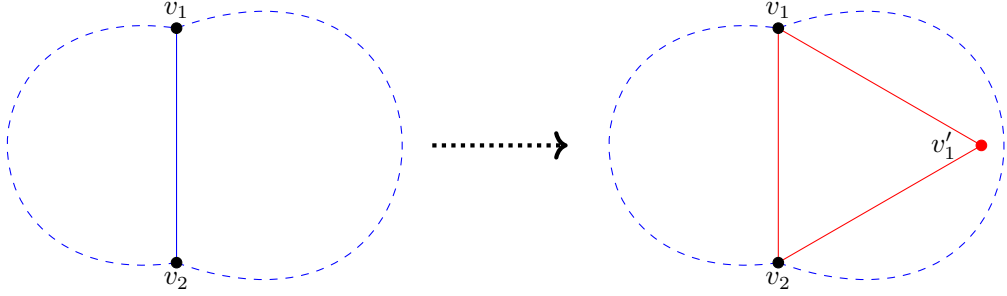
\begin{figure}
\centering
\begin{tikzpicture}
\draw[xshift=-4cm,blue](0,-1.56)--(0,1.56);
\draw[xshift=-4cm,blue,dashed](0,-1.56)..controls (-3,-2)and (-3,2)..(0,1.56);
\draw[xshift=4cm,blue,dashed](0,-1.56)..controls (-3,-2)and (-3,2)..(0,1.56);
\draw[xshift=4cm,blue,dashed](0,-1.56)..controls (4,-3)and (4,3)..(0,1.56);
\draw[xshift=-4cm,blue,dashed](0,-1.56)..controls (4,-3)and (4,3)..(0,1.56);
\draw[ultra thick,dotted,->](-0.6cm,0)--(1.2cm,0);
\draw[xshift=4.9cm,red](120:1.8)--(-120:1.8)--(0:1.8)--cycle;
\df (-4,-1.56)circle(2pt)node[anchor=90]{$v_2$};
\df (-4,1.56)circle(2pt)node[anchor=-90]{$v_1$};
\df (4,-1.56)circle(2pt)node[anchor=90]{$v_2$};
\df (4,1.56)circle(2pt)node[anchor=-90]{$v_1$};
\draw (6.5,0)node[anchor=0]{$v_1'$};
\draw[fill,red] (6.7,0)circle(2pt);
\end{tikzpicture}
\caption{\label{rp2}When $T=K_2$}
\end{figure}

\textbf{Case 1: }$T=K_2$.

We perform the replacement shown in Figure \ref{rp2}. In the left of the arrow, $v_1v_2$ is an edge in $G'$ and both faces adjacent to it, represented by blue curves, are $(8+)$-faces. In the right of the arrow, we delete edge $(v_1v_2)$ and attached a $T_3$ within these two $(8+)$-faces where $v_1,v_2$ serves as two vertices on the boundary of this $T_3$. All other vertices of this $T_3$ are new vertices added to the graph.

Now let's show property $(3)$ and $(4)$ are preserved. It's trivial that each vertex has degree at least $3$ in the new graph. For any two adjacent vertex within $T_3$, they have degree sum at least $7$. Since we didn't change the rest of the graph, the only possible contradiction to $(3)$ is $v_1$ or $v_2$ has some adjacent vertex not in the $T_3$ whose degree is $1$, which is a contradiction with $G'$ being $2$-connected. For property $(4)$, since there is no $C_7$ within the $T_3$ we added, the only possible $C_7$ that formed through this replacement must uses both vertices within the $T_3$ and outside the $T_3$. But all exterior faces of this $T_3$ are $(8+)$-faces, so any $C_7$ formed must contain a $(8+)$-face in its interior.

During this replacement, we deleted one edge and added $12$ edges and $4$ vertices. Therefore,
\begin{align*}
e(G')&> {18\over 7}v(G')-{48\over 7}\\
\Rightarrow e(G)&=e(G')+11\\
&> {18\over 7}(v(G')+4)-{48\over 7}+(11-{18\over 7}\cdot 4)\\
&> {18\over 7}v(G)-{48\over 7}.
\end{align*}

In the rest cases, the boundary of $T$ forms a single cycle. Assume $e(G')\geq {18\over 7}v(G')-{48\over 7}$. Let $c$ be the length of the cycle formed by the boundary of $T$. Take the induced plane subgraph of $T$ and add edges inside its boundary randomly to make its interior triangulated. Let this new subgraph be called $T'$. By Euler's formula,
\begin{align*}
f(T')&=e(T')-v(T')+2\\
2e(T')&=3(f(T')-1)+c\\
e(T')&=3v(T')-3-c.
\end{align*}

\textbf{Case 2:} $c=3$.

We perform the replacement shown in Figure \ref{rp3}.
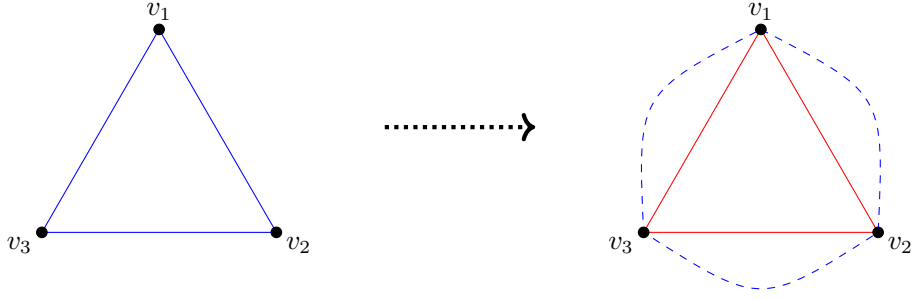
\begin{figure}
\centering
\begin{tikzpicture}
\draw[xshift=4cm,dashed,blue] (90:1.8)..controls(30:1.9)..(-30:1.8)..controls(-90:1.9)..(-150:1.8)..controls(150:1.9)..(90:1.8);
\draw[xshift=-4cm,blue](90:1.8)--(-30:1.8)--(-150:1.8)--cycle;
\draw[xshift=4cm,red](90:1.8)--(-30:1.8)--(-150:1.8)--cycle;
\draw[ultra thick,dotted,->](-1cm,0.5cm)--(1cm,0.5cm);
\draw[fill,xshift=-4cm](90:1.8)node[anchor=-90]{$v_1$}circle(2pt);
\draw[fill,xshift=-4cm](-30:1.8)node[anchor=150]{$v_2$}circle(2pt);
\draw[fill,xshift=-4cm](-150:1.8)node[anchor=30]{$v_3$}circle(2pt);

\draw[fill,xshift=4cm](90:1.8)node[anchor=-90]{$v_1$}circle(2pt);
\draw[fill,xshift=4cm](-30:1.8)node[anchor=150]{$v_2$}circle(2pt);
\draw[fill,xshift=4cm](-150:1.8)node[anchor=30]{$v_3$}circle(2pt);

\end{tikzpicture}
\caption{\label{rp3}When boundary of $T$ forms a $C_3$}
\end{figure}
In the left of the arrow, triangle $v_1v_2v_3$ is the boundary of $T$. There might be other vertices of $T$. We delete every vertex of $T$ that is not on its boundary, and also delete every edge on its boundary. Then we add a $T_3$, represented by the red triangle on the right of the arrow, whose boundary is $v_1v_2v_3$. The blue dashed lines at the right of the arrow is to show the replacement happens ``inside" triangle $v_1v_2v_3$ at the left of the arrow, so after the replacement the graph is still plane graph. Denote the graph after the replacement as $G$. Suppose $e(G')>{18\over 7}v(G')-{48\over 7}$, we will show $e(G)>{18\over 7}v(G)-{48\over 7}$.
\begin{align*}
e(G')&>{18\over 7}v(G')-{48\over 7}\\
v(G)&=v(G')+6-v(T')\\
e(G)&\geq e(G')+12-e(T')\\
e(G)&>{18\over 7}v(G)-{48\over 7}-{18\over 7}(6-v(T'))+(12-e(T')).
\end{align*}
Remember that $3\leq v(T')\leq 6$. When $v(T')=6$, $-{18\over 7}(6-v(T'))+(12-e(T'))=12-e(T')=12-(3v(T')-3-c)=0$. When $v(T')<6$,
${12-e(T')\over 6-v(T')}={12-3v(T')+3+c\over 6-v(T')}={18-3v(T')\over 6-v(T')}=3> {18\over 7}$.

\textbf{Case 3:} $c=4$.

We perform the replacement shown in Figure \ref{rp4}.
\begin{figure}
\centering
\begin{tikzpicture}
\draw[xshift=-4cm,blue](45:2.6)--(135:2.6)--(-135:2.6)--(-45:2.6)--cycle;
\draw[ultra thick,dotted,->](-1cm,0.5cm)--(1cm,0.5cm);
\draw[xshift=4cm,blue,dashed](45:2.6)..controls(90:2.8)..(135:2.6)..controls(180:2.8)..(-135:2.6)..controls(-90:2.8)..(-45:2.6)..controls(0:2.8)..(45:2.6);
\draw[xshift=4cm,red](45:2.6)--(135:2.6)--(0,0)--cycle;
\draw[xshift=4cm,red](-45:2.6)--(-135:2.6)--(0,0)--cycle;
\draw[fill,xshift=-4cm](45:2.6)node[anchor=-135]{$v_1$}circle(2pt);
\draw[fill,xshift=-4cm](135:2.6)node[anchor=-45]{$v_2$}circle(2pt);
\draw[fill,xshift=-4cm](-135:2.6)node[anchor=45]{$v_3$}circle(2pt);
\draw[fill,xshift=-4cm](-45:2.6)node[anchor=135]{$v_4$}circle(2pt);

\draw[fill,xshift=4cm](45:2.6)node[anchor=-135]{$v_1$}circle(2pt);
\draw[fill,xshift=4cm](135:2.6)node[anchor=-45]{$v_2$}circle(2pt);
\draw[fill,xshift=4cm](-135:2.6)node[anchor=45]{$v_3$}circle(2pt);
\draw[fill,xshift=4cm](-45:2.6)node[anchor=135]{$v_4$}circle(2pt);
\draw[xshift=4cm](0,0)node[anchor=-90]{$v_1'$};
\draw[fill,red,xshift=4cm](0,0)circle(2pt);

\end{tikzpicture}
\caption{\label{rp4}When boundary of $T$ forms a $C_4$}
\end{figure}
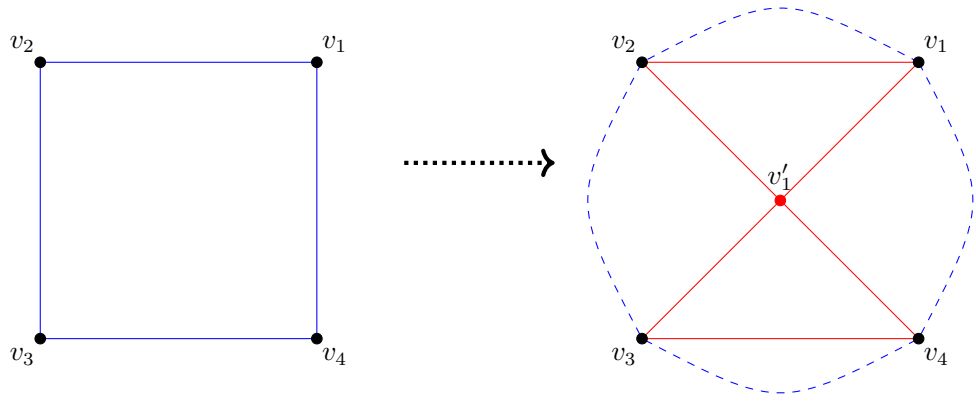
Similar to Case 2,
\begin{align*}
v(G)&=v(G')+11-v(T')\\
e(G)&=e(G')+24-e(T')\\
e(G)&>{18\over 7}v(G)-{48\over 7}-{18\over 7}(11-v(T'))+(24-e(T'))\\
{24-e(T')\over 11-v(T')}&={24-3v(T')+3+c\over 11-v(T')}\\
&={31-3v(T')\over 11-v(T')}\\
&\geq {31-3\cdot 6\over 11-6}={13\over 5}>{18\over 7}.
\end{align*}

\textbf{Case 4:} $c=5$.

We perform the replacement shown in Figure \ref{rp5}.
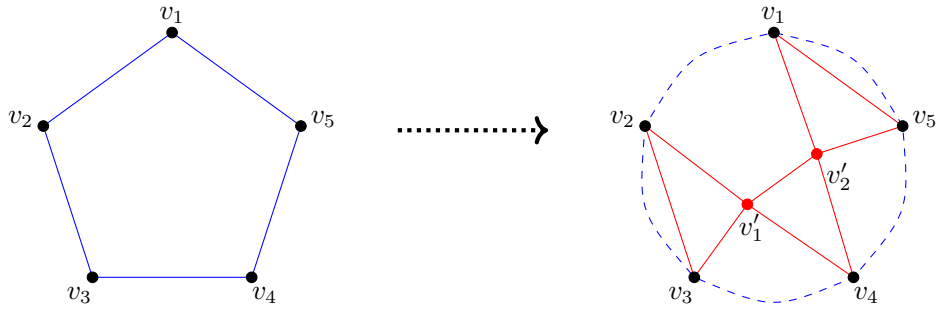
\begin{figure}
\centering
\begin{tikzpicture}
\draw[blue,xshift=-4cm](18:1.8)--(90:1.8)--(162:1.8)--(234:1.8)--(306:1.8)--cycle;
\draw[blue,xshift=4cm,dashed](18:1.8)..controls(54:1.9)..(90:1.8)..controls(126:1.9)..(162:1.8)..controls(198:1.9)..(234:1.8)..controls(270:1.9)..(306:1.8)..controls(-18:1.9)..(18:1.8);
\draw[red,xshift=4cm](162:1.8)--(234:0.6)--(234:1.8)--cycle;
\draw[red,xshift=4cm](18:1.8)--(90:1.8)--(18:0.6)--cycle;
\draw[red,xshift=4cm](234:0.6)--(18:0.6)--(306:1.8)--cycle;
\draw[ultra thick,dotted,->](-1cm,0.5cm)--(1cm,0.5cm);
\draw[fill,xshift=-4cm](18:1.8)node[anchor=198]{$v_5$}circle(2pt);
\draw[fill,xshift=-4cm](90:1.8)node[anchor=-90]{$v_1$}circle(2pt);
\draw[fill,xshift=-4cm](162:1.8)node[anchor=-18]{$v_2$}circle(2pt);
\draw[fill,xshift=-4cm](234:1.8)node[anchor=54]{$v_3$}circle(2pt);
\draw[fill,xshift=-4cm](306:1.8)node[anchor=126]{$v_4$}circle(2pt);

\draw[fill,xshift=4cm](18:1.8)node[anchor=198]{$v_5$}circle(2pt);
\draw[fill,xshift=4cm](90:1.8)node[anchor=-90]{$v_1$}circle(2pt);
\draw[fill,xshift=4cm](162:1.8)node[anchor=-18]{$v_2$}circle(2pt);
\draw[fill,xshift=4cm](234:1.8)node[anchor=54]{$v_3$}circle(2pt);
\draw[fill,xshift=4cm](306:1.8)node[anchor=126]{$v_4$}circle(2pt);

\draw[red,xshift=4cm,fill](234:0.6)circle(2pt);
\draw[xshift=4cm](234:0.6)node[anchor=100]{$v_1'$};
\draw[red,xshift=4cm,fill](18:0.6)circle(2pt);
\draw[xshift=4cm](18:0.6)node[anchor=135]{$v_2'$};

\end{tikzpicture}
\caption{\label{rp5}When boundary of $T$ forms a $C_5$}
\end{figure}
Similar to Case 2,
\begin{align*}
v(G)&=v(G')+16-v(T')\\
e(G)&\geq e(G')+36-e(T')\\
e(G)&>{18\over 7}v(G)-{48\over 7}-{18\over 7}(16-v(T'))+(36-e(T'))\\
{36-e(T')\over 16-v(T')}&={36-3v(T')+3+c\over 16-v(T')}\\
&={44-3v(T')\over 16-v(T')}\\
&\geq {44-3\cdot 6\over 16-6}={26\over 10}>{18\over 7}.
\end{align*}

\textbf{Case 5:} $c=6$.

We perform the replacement shown in Figure \ref{rp6}.
\begin{figure}
\centering
\begin{tikzpicture}
\draw[xshift=4cm,red](70:1.2)--(60:2.1)--(120:2.1)--cycle;
\draw[xshift=4cm,red](70:1.2)--(-30:0.6)--(0:2.1)--cycle;
\draw[xshift=4cm,red](-130:1.2)--(-30:0.6)--(-60:2.1)--cycle;
\draw[xshift=4cm,red](-130:1.2)--(180:2.1)--(-120:2.1)--cycle;
\draw[ultra thick,dotted,->](-1cm,0.5cm)--(1cm,0.5cm);
\draw[xshift=4cm,dashed,blue](0:2.1)..controls(30:2.2)..(60:2.1)..controls(90:2.2)..(120:2.1)..controls(150:2.2)..(180:2.1)..controls(210:2.2)..(240:2.1)..controls(270:2.2)..(300:2.1)..controls(-30:2.2)..(0:2.1);
\draw[blue,xshift=-4cm](0:2.1)--(60:2.1)--(120:2.1)--(180:2.1)--(-120:2.1)--(-60:2.1)--cycle;
\draw[fill,xshift=-4cm](60:2.1)node[anchor=-120]{$v_2$}circle(2pt);
\draw[fill,xshift=-4cm](0:2.1)node[anchor=180]{$v_1$}circle(2pt);
\draw[fill,xshift=-4cm](120:2.1)node[anchor=300]{$v_3$}circle(2pt);
\draw[fill,xshift=-4cm](180:2.1)node[anchor=0]{$v_4$}circle(2pt);
\draw[fill,xshift=-4cm](-120:2.1)node[anchor=60]{$v_5$}circle(2pt);
\draw[fill,xshift=-4cm](-60:2.1)node[anchor=120]{$v_6$}circle(2pt);

\draw[fill,xshift=4cm](60:2.1)node[anchor=-120]{$v_2$}circle(2pt);
\draw[fill,xshift=4cm](0:2.1)node[anchor=180]{$v_1$}circle(2pt);
\draw[fill,xshift=4cm](120:2.1)node[anchor=300]{$v_3$}circle(2pt);
\draw[fill,xshift=4cm](180:2.1)node[anchor=0]{$v_4$}circle(2pt);
\draw[fill,xshift=4cm](-120:2.1)node[anchor=60]{$v_5$}circle(2pt);
\draw[fill,xshift=4cm](-60:2.1)node[anchor=120]{$v_6$}circle(2pt);
\draw[red,xshift=4cm,fill](70:1.2)circle(2pt);
\draw[xshift=4cm](70:1.2)node[anchor=-80]{$v_3'$};
\draw[red,xshift=4cm,fill](-130:1.2)circle(2pt);
\draw[xshift=4cm](-130:1.2)node[anchor=20]{$v_1'$};
\draw[red,xshift=4cm,fill](-30:0.6)circle(2pt);
\draw[xshift=4cm](-30:0.6)node[anchor=-30]{$v_2'$};

\end{tikzpicture}
\caption{\label{rp6}When boundary of $T$ forms a $C_6$}
\end{figure}
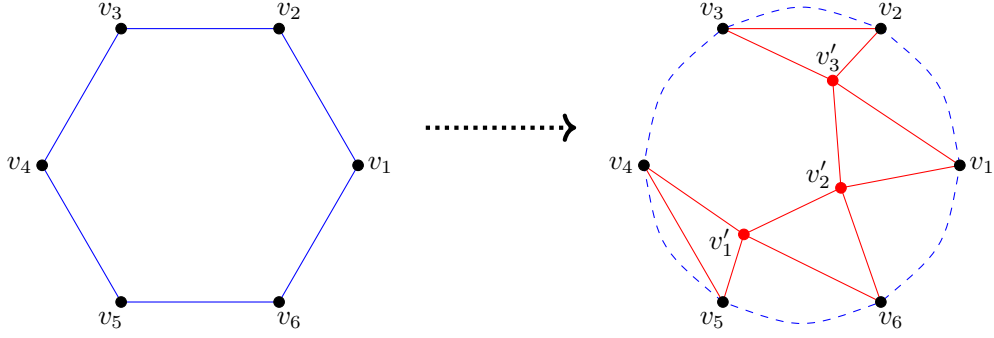
Similar to Case 2,
\begin{align*}
v(G)&=v(G')+21-v(T')\\
e(G)&\geq e(G')+48-e(T')\\
e(G)&>{18\over 7}v(G)-{48\over 7}-{18\over 7}(16-v(T'))+(36-e(T'))\\
{48-e(T')\over 21-v(T')}&={48-3v(T')+3+c\over 21-v(T')}\\
&={57-3v(T')\over 21-v(T')}\\
&\geq {57-3\cdot 6\over 21-6}={39\over 15}>{18\over 7}.
\end{align*}

\end{proof}

\begin{lemma}\label{adjacencylemma}
If $G$ is $2$-connected, $C_7$-free plane graph and $G\neq T_3$, then every incident face of any $T_3$ in $G$ has to be an $(8+)$-face.
\end{lemma}
\begin{figure}
\centering
\begin{tikzpicture}

\draw[red] (90:2)..controls(60:2) and (0:2)..(-30:2);
\draw[red] (-30:2)..controls(45:6) and (135:6)..(-150:2);
\draw[blue] (90:2)--(-30:2)--(-150:2)--cycle;
\draw[fill](90:2)circle(2pt)node[right]{$a$};
\draw[fill](-30:2)circle(2pt)node[below]{$c$};
\draw[fill](-150:2)circle(2pt)node[below]{$b$};
\draw[red] (30:1)node[anchor=-150]{$l_1$};
\draw[red] (90:2.3)node[anchor=-90]{$l_2$};
\end{tikzpicture}
\caption{\label{fc}When $f$ incident to $c$}
\end{figure}
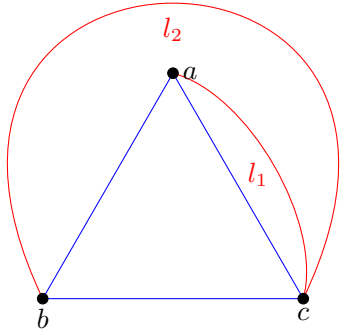
\begin{proof}
Let $f$ be an exterior face of a $T_3$. For contradiction assume $|f|<8$. We know $f$ uses at most $2$ edges of $T_3$, otherwise we either have $G=T_3$ or $G$ being not $2$-connected. Let the vertices on the boundary of this $T_3$ be $a,b,c$. If $f$ uses two edges of $T_3$, WLOG $(ab),(bc)$, then $f-\{(ab),(bc)\}$ is a path with length at most $5$ and at least $2$ between $a,c$ and doesn't use any vertex of $T_3$ other than $a,c$. Since within $T_3$, $a,c$ has path of any length between $1$ and $5$, we can easily find a $7$-cycle.

On the other hand, if $f$ uses only one edge of $T_3$, say $(ab)$, then $f$ has be incident to the other vertex $c$ (Figure \ref{fc}). If $f$ is not incident to $c$, then $f-\{(ab)\}$ is a path between $a,b$ that uses no other vertex of $T_3$ with length between $2$ and $5$ and we can find a $7$-cycle. Suppose $f$ is incident to $c$, then $f$ must be the concatenation of three parts: edge $(a,b)$, path $l_2$, and path $l_1$. Observe that $l_2$ is a path with length at most $4$ and at least $2$ between $a,c$ and doesn't use any vertex of $T_3$ other than $a,c$. Thus we can find a $7$-cycle.
\end{proof}

\begin{lemma}
\label{smallblockslemma}
Let $G'$ be a $2$-connected, $C_7$-free, plane graph on $n$ ($n\geq 8$) vertices with $\delta(G')\geq 3$ and with any adjacent vertices having total degree at least $7$. Let $G$ be the plane graph obtained from $G'$ by doing the replacement described in Lemma \ref{replacement}. Then,
\begin{equation*}
\sum_{B\in S(G)}f_B \leq {41\over 72}\sum_{B\in S(G)}e_B^\partial+{11\over 18}\sum_{B\in S(G)}e_B^{int}.
\end{equation*}

\end{lemma}

\begin{proof}
We will do case analysis. Notice that when a face is adjacent to only one triangular block that is not a trivial block($K_2)$, we can uniquely assign that face, together with the trivial blocks on its boundary, to that non-trivial triangular block and do the calculation together.

\textbf{Case 1: }$B$ is $B_2$.
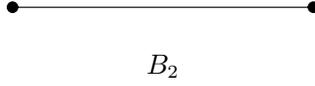
\begin{figure}
\centering
\begin{tikzpicture}
\draw[fill](-2,0)circle(2pt)--(2,0)circle(2pt);
\draw (0,-0.5) node[below] {$B_{2}$};
\end{tikzpicture}
\caption{\label{b2}Triangular block on two vertices(trivial block)}
\end{figure}

The faces adjacent to $B_2$ cannot have length $2$ because $G$ is connected and $n\geq 7$. It cannot have length $3$ because otherwise we would end up with a larger triangular block by Definition \ref{tb}. Hence,
\begin{align*}
f_B&\leq {1\over 4}+{1\over 4}={1\over 2}\\
{41\over 72}e_B^\partial+{11\over 18}e_B^{int}&\geq {41\over 72}\\
f_B&\leq {41\over 72}e_B^\partial+{11\over 18}e_B^{int}.
\end{align*}

\textbf{Case 2: }$B$ is $B_3$.
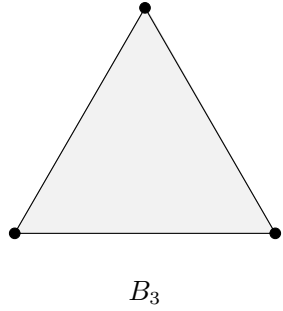
\begin{figure}
\centering
\begin{tikzpicture}
\draw[fill,color=gray!10] (-30:2)--(90:2)--(210:2)--cycle;
\draw[fill] (-30:2)circle(2pt)--(90:2)circle(2pt)--(210:2)circle(2pt)--(-30:2);
\draw (0,-1.5) node[below] {$B_3$};
\end{tikzpicture}
\caption{\label{b3}Triangular block on three vertices}
\end{figure}

Since we assume there is no small block sets with less than $7$ vertices in $G$, at least one edge of $B_3$ is an interior edge. Thus,
\begin{align*}
f_B&\leq 1+{3\over 4}={7\over 4}\\
{41\over 72}e_B^\partial+{11\over 18}e_B^{int}&\geq 2\cdot {41\over 72}+{11\over 18}={7\over 2}\\
f_B&\leq {41\over 72}e_B^\partial+{11\over 18}e_B^{int}.
\end{align*}

\textbf{Case 3: }$B$ is $B_{4,a}$.
\begin{figure}
\centering
\begin{tikzpicture}
\draw[fill,color=gray!10] (210:2)--(90:2)--(30:4)--(-30:2)--cycle;
\draw[fill](210:2)circle(2pt);
\draw[fill](90:2)circle(2pt);
\draw[fill](30:4)circle(2pt);
\draw[fill](-30:2)circle(2pt);
\draw (210:2)--(90:2)--(30:4)--(-30:2)--cycle;
\draw(90:2)--(-30:2);
\draw (0,-1.5) node[below] {$B_{4,a}$};
\end{tikzpicture}
\begin{tikzpicture}
\draw[fill,color=gray!10] (-30:2)--(90:2)--(210:2)--cycle;
\draw[fill] (-30:2)circle(2pt)--(90:2)circle(2pt)--(210:2)circle(2pt)--(-30:2);
\draw (0,-1.5) node[below] {$B_{4,b}$};
\draw[fill] (0,0)circle(2pt)--(-30:2);
\draw (0,0)--(90:2);
\draw (0,0)--(210:2);
\end{tikzpicture}

\caption{\label{b4}Triangular blocks on four vertices}
\end{figure}
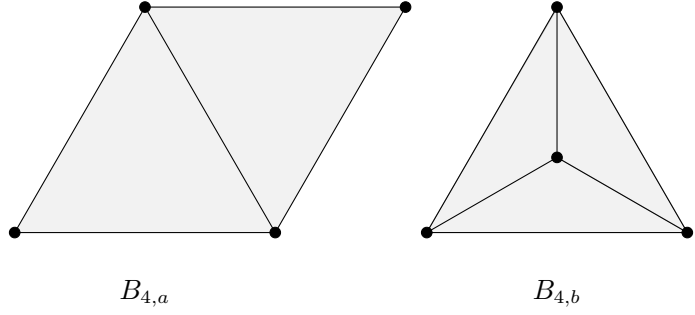

Similar to Case 2, at least one of the edge on the boundary of $B$ has to be an exterior edge. Hence,
\begin{align*}
2\leq e_B^{int}&\leq 5\\
f_B&\leq 2+{1\over 4}(e_B^{int}-1)+{1\over 8}e_B^\partial \\
&=2+{1\over 4}(e_B^{int}-1)+{1\over 8}(5-e_B^{int})={19+e_B^{int}\over 8}\\
{41\over 72}e_B^\partial+{11\over 18}e_B^{int}&={41\over 72}(5-e_B^{int})+{11\over 18}e_B^{int}={205-6e_B^{int}\over 72}+{e_B^{int}\over 8}.
\end{align*}
Since $e_B^{int}\leq 5$, we know ${205-6e_B^{int}\over 72}>{19\over 8}$. Hence ${205-6e_B^{int}\over 72}+{e_B^{int}\over 8}>{19+e_B^{int}\over 8}$ and $f_B\leq {41\over 72}e_B^\partial+{11\over 18}e_B^{int}.$

\textbf{Case 4: }$B$ is $B_{4,b}$.

If at most one of exterior edges of $B$ is adjacent to a face of length $4$, then
\begin{align*}
f_B&\leq 3+{1\over 4}+{1\over 5}(e_B^{int}-4)+{1\over 8}(6-e_B^{int})={3\over 40}e_B^{int}+{16\over 5}\\
{41\over 72}e_B^\partial+{11\over 18}e_B^{int}&={41\over 72}(6-e_B^{int})+{11\over 18}e_B^{int}={1\over 24}e_B^{int}+{41\over 12}\\
4\leq e_B^{int}&\leq 6\\
\Rightarrow {3\over 40}e_B^{int}+{16\over 5}&<{1\over 24}e_B^{int}+{41\over 12}\\
\Rightarrow f_B&\leq {41\over 72}e_B^\partial+{11\over 18}e_B^{int}.
\end{align*}
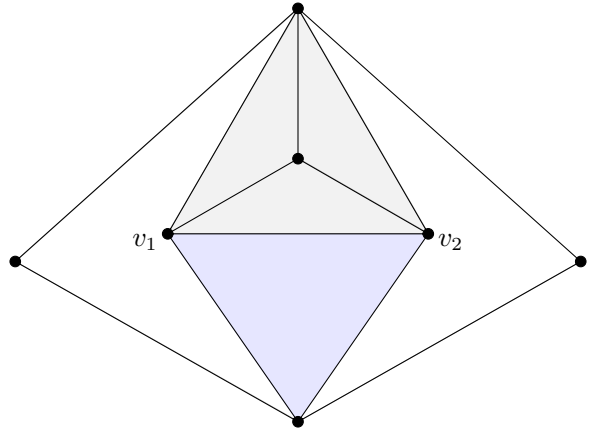
\begin{figure}
\centering
\begin{tikzpicture}
\draw[fill,color=gray!10] (-30:2)--(90:2)--(210:2)--cycle;
\draw[fill,color=blue!10] (-30:2)--(210:2)--(-90:3.5);
\draw[fill] (-30:2)circle(2pt)--(90:2)circle(2pt)--(210:2)circle(2pt)--(-30:2);

\draw[fill] (0,0)circle(2pt)--(-30:2);
\draw (0,0)--(90:2);
\draw (0,0)--(210:2);
\df (90:2)--(200:4)circle(2pt)--(-90:3.5)circle(2pt)--(-20:4)circle(2pt);
\draw (-20:4)--(90:2);
\draw (-30:2)node[anchor=160]{$v_2$}--(-90:3.5)--(210:2)node[anchor=20]{$v_1$};

\end{tikzpicture}
\caption{\label{b4b}When two boundary edges of $B_{4,b}$ are adjacent to some face of length $4$(up to symmetry)}
\end{figure}
Remember that $G$ is $2$-connected, every pair of adjacent vertices have degree sum at least $7$, $\delta(G)\geq 3$, and any $C_7$ of $G$ has to contain an $(8+)$-face in its interior. Therefore, if more than one of exterior edges of $B$ is adjacent to a face of length $4$, there is only one possible drawing, as shown in Figure \ref{b4b}. The blue area is not a face and there are other vertices inside it. Observe that in this case, edge $(v_1v_2)$ has to be adjacent to an $(8+)$-face in the blue area. Hence,
\begin{align*}
f_B&\leq 3+{2\over 4}+{1\over 8}={29\over 8}\\
{41\over 72}e_B^\partial+{11\over 18}e_B^{int}&={41\over 72}+{11\over 18}\cdot 5={29\over 8}\\
f_B&\leq {41\over 72}e_B^\partial+{11\over 18}e_B^{int}.
\end{align*}

Notice that there is another tricky graph of two boundary edges of $B_{4,b}$ adjacent to some face of length $4$ that could not exist. This graph appears because we allow $C_7$ with an $(8+)$-face inside. As shown in Figure \ref{b4n} where the blue area is not a face and contains $(8+)$-faces inside, the only possible $C_7$ must contain an $(8+)$-face. However, any of the red edge could not belong to a $T_3$ block, so they are all original edges of $G'$ and should not form a $C_7$.

\begin{figure}
\centering
\begin{tikzpicture}

\draw[fill,color=gray!10] (-30:2)--(90:2)--(210:2)--cycle;
\draw[fill,color=blue!10] (135:4)--(150:1.8)--(90:2)--cycle;
\draw[fill] (-30:2)circle(2pt)--(90:2)circle(2pt)--(210:2)circle(2pt)--(-30:2);
\draw (0,-1.5) node[below] {$B_{4,b}$};
\draw[fill] (0,0)circle(2pt)--(-30:2);
\draw (0,0)--(90:2);
\draw (0,0)--(210:2);
\df (-30:2)--(45:4)circle(2pt)--(135:4)circle(2pt)--(210:2);
\df(135:4)--(150:1.8)circle(2pt)--(90:2);
\draw (135:4)--(90:2);
\draw[color=red,thick](135:4)--(150:1.8)--(90:2)--(0,0)--(210:2)--(-30:2)--(45:4)--cycle;

\end{tikzpicture}
\caption{\label{b4n}An impossible drawing}
\end{figure}
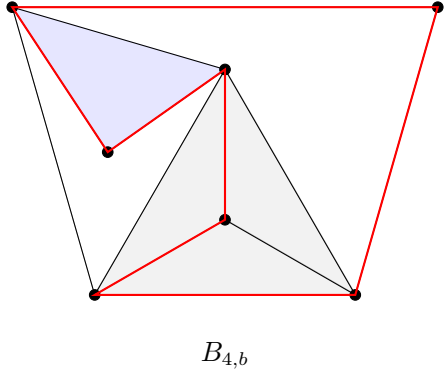
\textbf{Case 5: }$B$ is $B_{5,a}$.
\begin{figure}
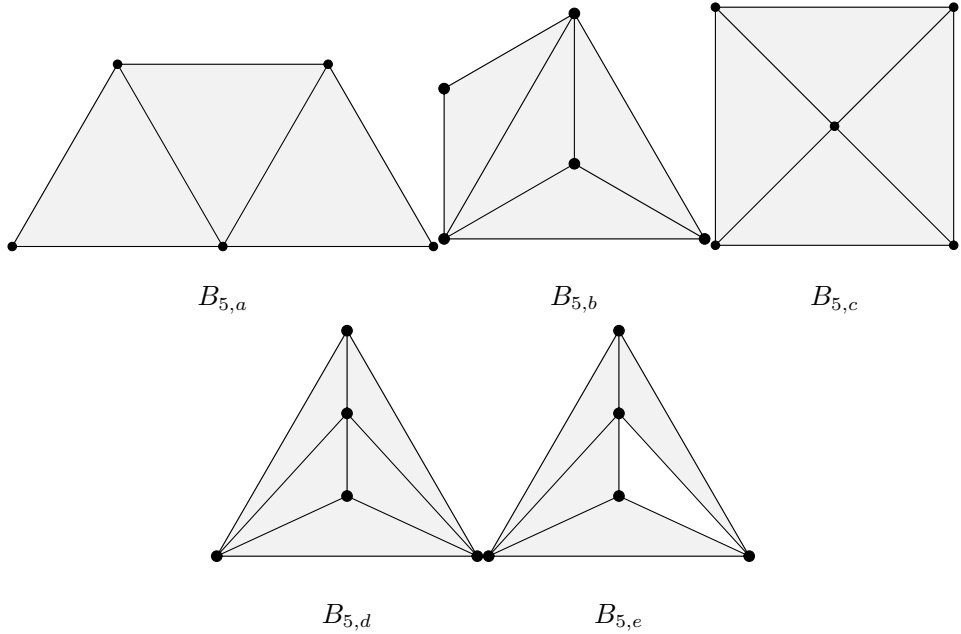

\centering
\bt[scale=0.8]
\dff (180:3.5)--(120:3.5)--(60:3.5)--(0:3.5)--cycle;
\df (0,0)circle(2pt)--(180:3.5)circle(2pt)--(120:3.5)circle(2pt)--(60:3.5)circle(2pt)--(0:3.5)circle(2pt)--(0,0);
\draw(120:3.5)--(0,0)--(60:3.5);
\draw (0,-0.5)node[below]{$B_{5,a}$};
\et
\bt
\draw[fill,color=gray!10] (-30:2)--(90:2)--(150:2)--(210:2)--cycle;
\draw[fill] (-30:2)circle(2pt)--(90:2)circle(2pt)--(210:2)circle(2pt)--(-30:2);
\df (90:2)--(150:2)circle(2pt)--(210:2);
\draw (0,-1.5) node[below] {$B_{5,b}$};
\draw[fill] (0,0)circle(2pt)--(-30:2);
\draw (0,0)--(90:2);
\draw (0,0)--(210:2);
\et
\bt[scale=0.8]
\dff (45:2.8)--(135:2.8)--(225:2.8)--(-45:2.8)--cycle;
\df (45:2.8)circle(2pt)--(135:2.8)circle(2pt)--(225:2.8)circle(2pt)--(-45:2.8)circle(2pt)--(45:2.8);
\df (45:2.8)--(0,0)circle(2pt)--(135:2.8);
\draw (225:2.8)--(0,0)--(-45:2.8);
\draw (0,-2.5) node[below]{$B_{5,c}$};

\et

\bt
\draw[fill,color=gray!10] (-30:2)--(90:2)--(210:2)--cycle;
\draw[fill] (-30:2)circle(2pt)--(90:2)circle(2pt)--(210:2)circle(2pt)--(-30:2);
\df (90:2)--(90:0.9)circle(2pt)--(90:-0.2)circle(2pt);
\draw (90:0.9)--(210:2)--(90:-0.2)--(-30:2)--cycle;
\draw (0,-1.5) node[below] {$B_{5,d}$};
\et
\bt
\draw[fill,color=gray!10] (-30:2)--(90:2)--(210:2)--cycle;
\draw[fill,color=white] (90:0.9)--(90:-0.2)--(-30:2)--cycle;
\draw[fill] (-30:2)circle(2pt)--(90:2)circle(2pt)--(210:2)circle(2pt)--(-30:2);
\df (90:2)--(90:0.9)circle(2pt)--(90:-0.2)circle(2pt);
\draw (90:0.9)--(210:2)--(90:-0.2)--(-30:2)--cycle;
\draw (0,-1.5) node[below] {$B_{5,e}$};
\et

\caption{\label{b5}Triangular blocks on five vertices}
\end{figure}
\begin{align*}
f_B&\leq 3+{1\over 4}(e_B^{int}-2)+{1\over 8}(7-e_B^{int})={27+e_B^{int}\over 8}\\
{41\over 72}e_B^\partial+{11\over 18}e_B^{int}&={41\over 72}e_B^{int}+{11\over 18}e_B^{int}={1\over 24}e_B^{int}+{287\over 72}.
\end{align*}
Since $e_B^{int}\leq 7$, we know ${1\over 24}e_B^{int}+{287\over 72}>{27+e_B^{int}\over 8}$. Hence $f_B\leq {41\over 72}e_B^\partial+{11\over 18}e_B^{int}.$

\textbf{Case 6: }$B$ is $B_{5,b}$.

\begin{figure}
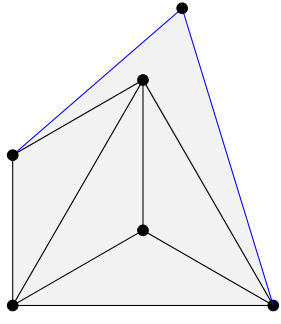

\centering

\bt
\draw[fill,color=gray!10] (-30:2)--(90:2)--(150:2)--(210:2)--cycle;
\dff (-30:2)--(80:3)--(150:2)--(90:2)--cycle;
\draw[fill] (-30:2)circle(2pt)--(90:2)circle(2pt)--(210:2)circle(2pt);
\draw[color=blue](-30:2)--(80:3)--(150:2);
\df(80:3)circle(2pt);
\df (90:2)--(150:2)circle(2pt);
\draw(150:2)--(210:2)--(-30:2);
\draw[fill] (0,0)circle(2pt)--(-30:2);
\draw (0,0)--(90:2);
\draw (0,0)--(210:2);
\et

\caption{\label{b5b}At least two boundary edge of $B_{5,b}$ adjacent to some face of length $4$}

\end{figure}
Given the assumptions of Lemma \ref{smallblockslemma}, if at least two boundary edges of $B_{5,b}$ is adjacent to some face of length $4$, then there is only one way to draw that(up to symmetry), as shown in Figure \ref{b5b}. Observe that under the assumptions of Lemma \ref{smallblockslemma}, each of the two blue edges has to bordor some face of length at least $6$, so we can unique assign them to $B$, together with the face of length $4$. In this case,
\begin{align*}
f_B&\leq 5+{1\over 8}e_B^\partial+{1\over 4}(4-e_B^\partial)=6-{1\over 8}e_B^\partial\\
{41\over 72}e_B^\partial+{11\over 18}e_B^{int}&={41\over 72}e_B^\partial+{11\over 18}(10-e_B^\partial)={55\over 9}-{1\over 24}e_B^\partial.
\end{align*}
Since $e_B^\partial\geq 0$, we know $6-{1\over 8}e_B^\partial\leq {55\over 9}-{1\over 24}e_B^\partial$ and $f_B\leq{41\over 72}e_B^\partial+{11\over 18}e_B^{int}$.

Now we can assume at most one edge of $B_{5,b}$ is adjacent to some face of length $4$. In this case,
\begin{align*}
f_B&\leq 4+{1\over 8}e_B^\partial+{1\over 4}+{1\over 5}(3-e_B^\partial)=6-{1\over 8}e_B^\partial={97\over 20}-{3\over 40}e_B^\partial\\
{41\over 72}e_B^\partial+{11\over 18}e_B^{int}&={41\over 72}e_B^\partial+{11\over 18}(8-e_B^\partial)={44\over 9}-{1\over 24}e_B^\partial.
\end{align*}
Since $e_B^\partial\geq 0$, we know ${97\over 20}-{3\over 40}e_B^\partial\leq {44\over 9}-{1\over 24}e_B^\partial$ and $f_B\leq{41\over 72}e_B^\partial+{11\over 18}e_B^{int}$.

\textbf{Case 7: }$B$ is $B_{5,c}$.

\begin{figure}
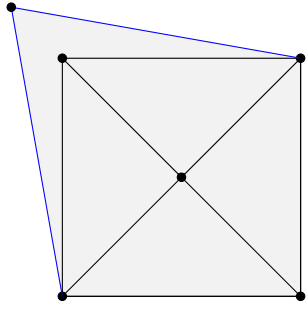

\centering

\bt[scale=0.8]

\dff (45:2.8)--(135:4)--(225:2.8)--(-45:2.8)--cycle;
\df (45:2.8)circle(2pt)--(135:2.8)circle(2pt)--(225:2.8)circle(2pt)--(-45:2.8)circle(2pt)--(45:2.8);
\draw[color=blue](45:2.8)--(135:4)--(225:2.8);
\df (135:4)circle(2pt);
\draw(45:2.8)--(-45:2.8)--(225:2.8);
\df (45:2.8)--(0,0)circle(2pt)--(135:2.8);
\draw (225:2.8)--(0,0)--(-45:2.8);
\et
\caption{\label{b5c}At least two boundary edge of $B_{5,c}$ adjacent to some face of length $4$}

\end{figure}
Given the assumptions of Lemma \ref{smallblockslemma}, if at least two boundary edges of $B_{5,c}$ are adjacent to some face of length $4$, then there is only one way to draw that(up to symmetry), as shown in Figure \ref{b5c}. Observe that under the assumptions of Lemma \ref{smallblockslemma}, each of the two blue edges is trivial block, so we can uniquely assign them to $B$, together with the face of length $4$. In this case,
\begin{align*}
f_B&\leq 5+{1\over 8}e_B^\partial+{1\over 4}(4-e_B^\partial)=6-{1\over 8}e_B^\partial\\
{41\over 72}e_B^\partial+{11\over 18}e_B^{int}&={41\over 72}e_B^\partial+{11\over 18}(10-e_B^\partial)={55\over 9}-{1\over 24}e_B^\partial.
\end{align*}
Since $e_B^\partial\geq 0$, we know $6-{1\over 8}e_B^\partial\leq {55\over 9}-{1\over 24}e_B^\partial$ and $f_B\leq{41\over 72}e_B^\partial+{11\over 18}e_B^{int}$.

Now we can assume at most one edge of $B_{5,b}$ is adjacent to some face of length $4$. In this case,
\begin{align*}
f_B&\leq 4+{1\over 8}e_B^\partial+{1\over 4}+{1\over 5}(3-e_B^\partial)=6-{1\over 8}e_B^\partial={97\over 20}-{3\over 40}e_B^\partial\\
{41\over 72}e_B^\partial+{11\over 18}e_B^{int}&={41\over 72}e_B^\partial+{11\over 18}(8-e_B^\partial)={44\over 9}-{1\over 24}e_B^\partial.
\end{align*}
Since $e_B^\partial\geq 0$, we know ${97\over 20}-{3\over 40}e_B^\partial\leq {44\over 9}-{1\over 24}e_B^\partial$ and $f_B\leq{41\over 72}e_B^\partial+{11\over 18}e_B^{int}$.

\textbf{Case 8: }$B$ is $B_{5,d}$.

\begin{figure}
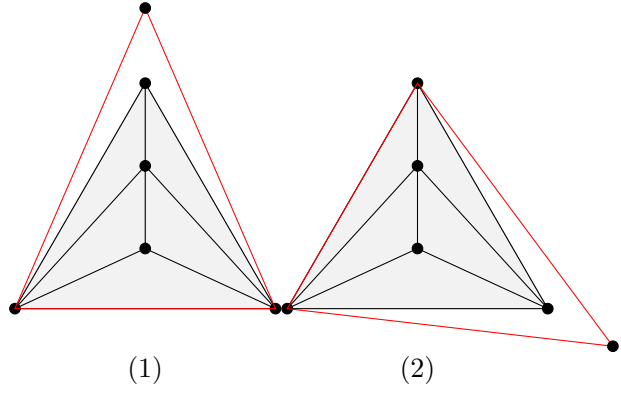

\centering
\bt
\draw[fill,color=gray!10] (-30:2)--(90:2)--(210:2)--cycle;
\draw[fill] (-30:2)circle(2pt)--(90:2)circle(2pt)--(210:2)circle(2pt)--(-30:2);
\df (90:2)--(90:0.9)circle(2pt)--(90:-0.2)circle(2pt);
\draw (90:0.9)--(210:2)--(90:-0.2)--(-30:2)--cycle;
\draw[color=red](90:3)--(210:2)--(-30:2)--cycle;
\draw(0,-1.5) node[below] {$(1)$};
\df(90:3)circle(2pt);
\et
\bt
\draw[fill,color=gray!10] (-30:2)--(90:2)--(210:2)--cycle;
\draw[fill] (-30:2)circle(2pt)--(90:2)circle(2pt)--(210:2)circle(2pt)--(-30:2);
\df (90:2)--(90:0.9)circle(2pt)--(90:-0.2)circle(2pt);
\draw (90:0.9)--(210:2)--(90:-0.2)--(-30:2)--cycle;
\draw[color=red](90:2)--(210:2)--(-30:3)--cycle;
\draw(0,-1.5) node[below] {$(2)$};
\df(-30:3)circle(2pt);
\et
\caption{\label{b5d}Two ways to draw $B_{5,d}$ while adjacent to a face of length $4$}
\end{figure}
Suppose $B$ is adjacent to some face of length $4$. To avoid $C_7$ without an $(8+)$-face inside, the only two possibly drawings are shown in Figure \ref{b5d}. Since $\delta(G)\geq 3$, we know the red edges are all adjacent to some $(8+)$-face. Thus in both drawings we have a small block set with $6$ vertices, which is not allowed. However, if $B$ is not adjacent to any face of length $4$, then it can only border $(8+)$-faces and we end up with a small block set with $5$ vertices and is also not allowed.

\textbf{Case 9: }$B$ is $B_{5,e}$.

Similar to Case 2,
\begin{align*}
0\leq e_B^\partial&\leq 6\\
f_G&\leq 4+{1\over 8}e_B^\partial+{1\over 4}(6-e_B^\partial)={11\over 2}-{1\over 8}e_B^\partial\\
{41\over 72}e_B^\partial+{11\over 18}e_B^{int}&={41\over 72}e_B^\partial+{11\over 18}(9-e_B^\partial)={11\over 2}-{1\over 24}e_B^\partial\geq {11\over 2}-{1\over 8}e_B^\partial\\
f_G&\leq{41\over 72}e_B^\partial+{11\over 18}e_B^{int}.
\end{align*}

\textbf{Case 10: }$B$ is $B_{6,a}$ or $B_{6,b}$.

\begin{figure}
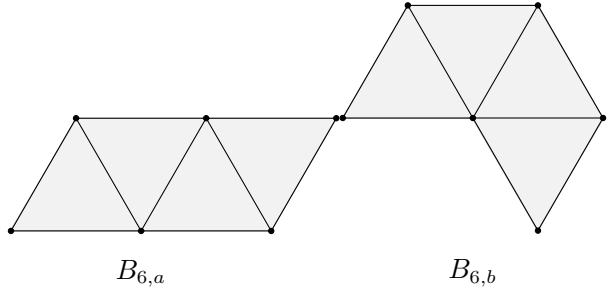

\centering
\bt[scale=0.5]
\coordinate(a) at(-3.46,0);
\coordinate(b) at(-1.73,3);
\coordinate(c)at(1.73,3);
\coordinate(d)at(5.19,3);
\coordinate(e)at(3.46,0);
\coordinate(f)at(0,0);
\dff (a)--(b)--(d)--(e)--cycle;
\df (a)circle(2pt)--(b)circle(2pt)--(c)circle(2pt)--(d)circle(2pt)--(e)circle(2pt)--(f)circle(2pt)--(a);
\draw (e)--(c)--(f)--(b);
\draw (0,-0.5) node[below]{$B_{6,a}$};
\et
\bt[scale=0.5]
\coordinate(a) at (180:3.46);
\coordinate(b) at (120:3.46);
\coordinate(c) at (60:3.46);
\coordinate(d) at (0:3.46);
\coordinate(e) at (-60:3.46);
\coordinate(f) at (0,0);
\dff (a)--(b)--(c)--(d)--(e)--(f)--cycle;
\df (a)circle(2pt)--(b)circle(2pt)--(c)circle(2pt)--(d)circle(2pt)--(e)circle(2pt)--(f)circle(2pt)--(a);
\draw (b)--(f)--(c);
\draw (f)--(d);
\draw (0,-3.5) node[below]{$B_{6,b}$};
\et
\caption{\label{b6b6}When $B$ has $6$ vertices and boundary forms a $C_6$}
\end{figure}
Similar to Case 9,

\begin{align*}
0\leq e_B^\partial&\leq 6\\
f_G&\leq 4+{1\over 8}e_B^\partial+{1\over 4}(6-e_B^\partial)={11\over 2}-{1\over 8}e_B^\partial\\
{41\over 72}e_B^\partial+{11\over 18}e_B^{int}&={41\over 72}e_B^\partial+{11\over 18}(9-e_B^\partial)={11\over 2}-{1\over 24}e_B^\partial\geq {11\over 2}-{1\over 8}e_B^\partial\\
f_G&\leq{41\over 72}e_B^\partial+{11\over 18}e_B^{int}.
\end{align*}

\textbf{Case 11: }$B$ is $B_{6,c},B_{6,d},B_{6,e},$ or $B_{6,f}$.

\begin{figure}
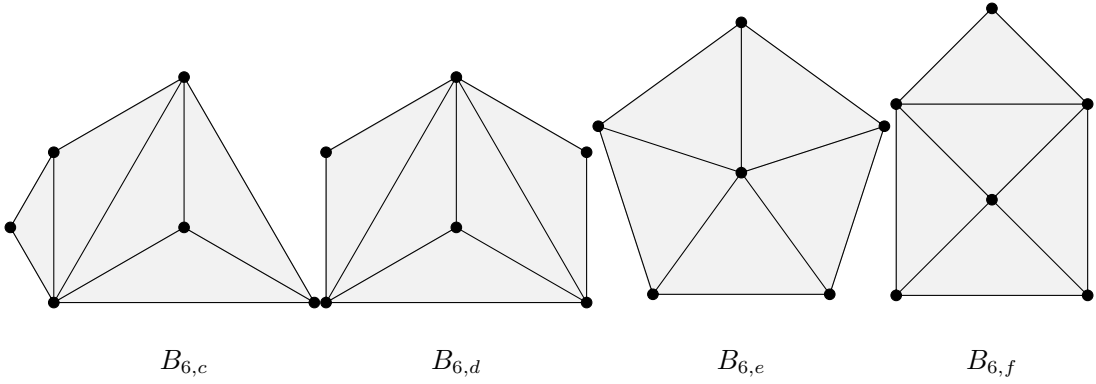

\centering
\bt
\coordinate(a)at(180:2.3094);
\coordinate(b)at(150:2);
\coordinate(c)at(90:2);
\coordinate(d)at(-30:2);
\coordinate(e)at(-150:2);
\coordinate(f)at(0,0);
\dff (a)--(b)--(c)--(d)--(e)--cycle;
\df (a)circle(2pt)--(b)circle(2pt)--(c)circle(2pt)--(d)circle(2pt)--(e)circle(2pt)--(a);
\df (d)--(f)circle(2pt)--(e);
\draw(b)--(e)--(c)--(f);
\draw (0,-1.5) node[below]{$B_{6,c}$};
\et
\bt
\coordinate(a)at(30:2);
\coordinate(b)at(150:2);
\coordinate(c)at(90:2);
\coordinate(d)at(-30:2);
\coordinate(e)at(-150:2);
\coordinate(f)at(0,0); 	
\dff (b)--(c)--(a)--(d)--(e)--cycle;
\df (b)circle(2pt)--(c)circle(2pt)--(a)circle(2pt)--(d)circle(2pt)--(e)circle(2pt)--(b);
\df (f)circle(2pt)--(e);
\draw (e)--(c)--(f)--(d)--(c);
\draw (0,-1.5) node[below]{$B_{6,d}$};
\et
\bt
\coordinate(a)at(18:2);
\coordinate(b)at(90:2);
\coordinate(c)at(162:2);
\coordinate(d)at(234:2);
\coordinate(e)at(306:2);
\coordinate(f)at(0,0);
\dff (a)--(b)--(c)--(d)--(e)--cycle;
\df (a)circle(2pt)--(b)circle(2pt)--(c)circle(2pt)--(d)circle(2pt)--(e)circle(2pt)--(a);
\df (a)--(f)circle(2pt)--(b);
\draw (e)--(f)--(c);
\draw (f)--(d);
\draw (0,-2.23)node[below]{$B_{6,e}$};
\et
\bt
\dff (45:1.8)--(135:1.8)--(225:1.8)--(-45:1.8)--cycle;
\dff (45:1.8)--(135:1.8)--(90:2.54558)--cycle;
\df (45:1.8)circle(2pt)--(135:1.8)circle(2pt)--(225:1.8)circle(2pt)--(-45:1.8)circle(2pt)--(45:1.8);
\df (45:1.8)--(0,0)circle(2pt)--(135:1.8);
\draw (225:1.8)--(0,0)--(-45:1.8);
\df(45:1.8)--(90:2.54558)circle(2pt)--(135:1.8);
\draw (0,-1.87) node[below]{$B_{6,f}$};
\et

\caption{\label{b6b5}When $B$ has $6$ vertices with boundary forms a $C_5$}
\end{figure}
In this case,
\begin{align*}
f_B&\leq 5+{1\over 8}e_B^\partial+{1\over 4}(5-e_B^\partial)={25\over 4}-{1\over 8}e_B^\partial\\
{41\over 72}e_B^\partial+{11\over 18}e_B^{int}&={41\over 72}e_B^\partial+{11\over 18}(10-e_B^\partial)={55\over 9}-{1\over 24}e_B^\partial.
\end{align*}
Therefore, when $e_B^\partial\geq 2$, we have $f_B\leq {41\over 72}e_B^\partial+{11\over 18}e_B^{int}$. It suffices to show that $B$ has at least two edges adjacent to some $(8+)$-face.

Remember that $\delta(G)\geq 3$, each pair of adjacent vertices have degree sum at least $7$, and there is no $C_7$ without an $(8+)$-face inside. The key observation is that, if $B$ is adjacent to a face with length less than $8$ and the edges and vertices shared form a path on the boundary of $B$, then there is only one possibility to draw such a face, as shown in Figure \ref{b5p1}. However, in such embedding, the red edges all need to be adjacent to some $(8+)$-face, which leaves us a small block set with $6$ vertices and is not allowed.

Now assume every face of size less than $8$ incident to $B$ shares edges and vertices that do not form a path. If at most one of boundary edges of $B$ is adjacent to some $(8+)$-face, we know $B$ is adjacent to at least one face with size less than $8$. The boundary edges of $B$ not used by this face will be cut into at least $2$ paths and non-triangular faces adjacent to each of these paths has to be disjoint from each other. If a path has length $1$ then it has to be adjacent to an $(8+)$-face. If the path is longer than one, by the key observation above, any faces adjacent to it has to use disjoint parts on that path, which would leave us with a shorter path unused. Therefore there are at least $2$ edges of $B$ adjacent to some $(8+)$-face.

\begin{figure}
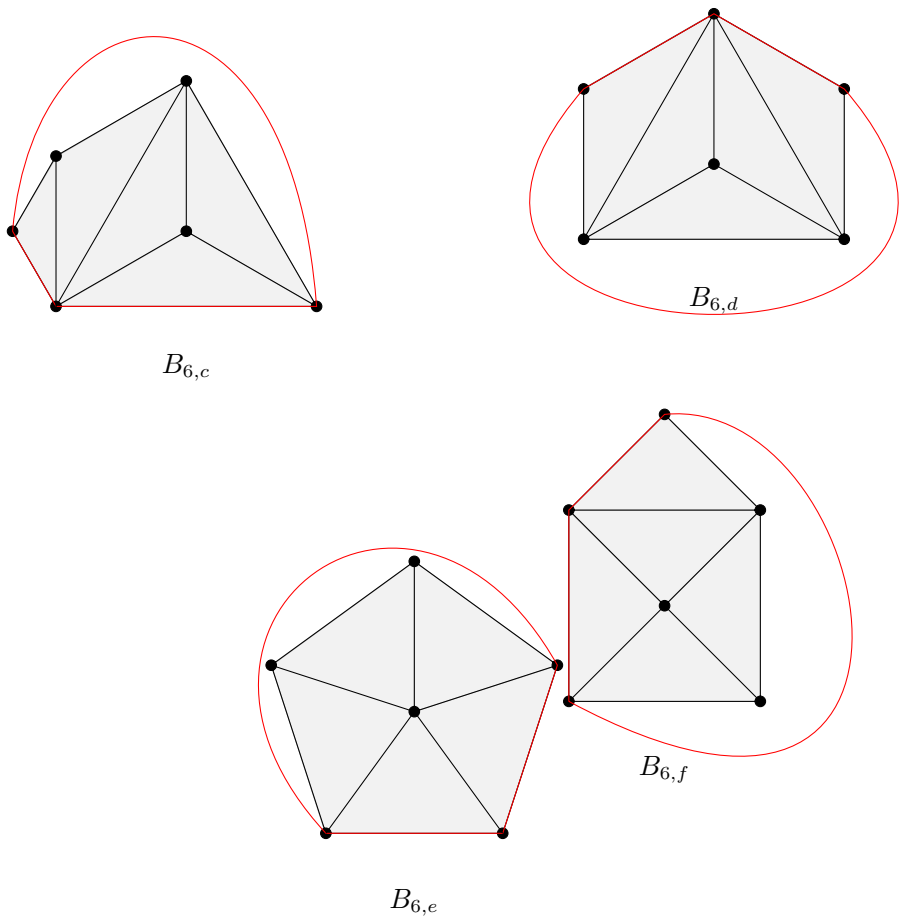

\centering
\bt
\coordinate(a)at(180:2.3094);
\coordinate(b)at(150:2);
\coordinate(c)at(90:2);
\coordinate(d)at(-30:2);
\coordinate(e)at(-150:2);
\coordinate(f)at(0,0);
\dff (a)--(b)--(c)--(d)--(e)--cycle;
\df (a)circle(2pt)--(b)circle(2pt)--(c)circle(2pt)--(d)circle(2pt)--(e)circle(2pt)--(a);
\df (d)--(f)circle(2pt)--(e);
\draw[color=red] (180:2.3094)..controls(120:4)and(70:4)..(-30:2)--(-150:2)--(180:2.3094);
\draw(b)--(e)--(c)--(f);
\draw (0,-1.5) node[below]{$B_{6,c}$};
\et
\bt
\coordinate(a)at(30:2);
\coordinate(b)at(150:2);
\coordinate(c)at(90:2);
\coordinate(d)at(-30:2);
\coordinate(e)at(-150:2);
\coordinate(f)at(0,0); 	
\dff (b)--(c)--(a)--(d)--(e)--cycle;

\df (b)circle(2pt)--(c)circle(2pt)--(a)circle(2pt)--(d)circle(2pt)--(e)circle(2pt)--(b);
\df (f)circle(2pt)--(e);
\draw (e)--(c)--(f)--(d)--(c);
\draw[color=red](150:2)..controls(-150:6)and(-30:6)..(30:2)--(90:2)--(150:2);
\draw (0,-1.5) node[below]{$B_{6,d}$};
\et

\bt
\coordinate(a)at(18:2);
\coordinate(b)at(90:2);
\coordinate(c)at(162:2);
\coordinate(d)at(234:2);
\coordinate(e)at(306:2);
\coordinate(f)at(0,0);
\dff (a)--(b)--(c)--(d)--(e)--cycle;
\df (a)circle(2pt)--(b)circle(2pt)--(c)circle(2pt)--(d)circle(2pt)--(e)circle(2pt)--(a);
\df (a)--(f)circle(2pt)--(b);
\draw[color=red](18:2)..controls(90:4)and(162:4)..(234:2)--(306:2)--(18:2);
\draw (e)--(f)--(c);
\draw (f)--(d);
\draw (0,-2.23)node[below]{$B_{6,e}$};
\et 	 	
\bt
\dff (45:1.8)--(135:1.8)--(225:1.8)--(-45:1.8)--cycle;
\dff (45:1.8)--(135:1.8)--(90:2.54558)--cycle;
\df (45:1.8)circle(2pt)--(135:1.8)circle(2pt)--(225:1.8)circle(2pt)--(-45:1.8)circle(2pt)--(45:1.8);
\df (45:1.8)--(0,0)circle(2pt)--(135:1.8);
\draw (225:1.8)--(0,0)--(-45:1.8);
\df(45:1.8)--(90:2.54558)circle(2pt)--(135:1.8);
\draw[color=red](225:1.8)..controls(-45:6)and(45:4)..(90:2.54558)--(135:1.8)--(225:1.8);
\draw (0,-1.87) node[below]{$B_{6,f}$};

\et
\caption{\label{b5p1}When not $(8+)$-face adjacent to $B$ shares only a path with $B$}
\end{figure}

\textbf{Case 12: }$B$ is $B_{6,g},B_{6,h},B_{6,i},B_{6,j}$, or $B_{6,k}$.

\begin{figure}
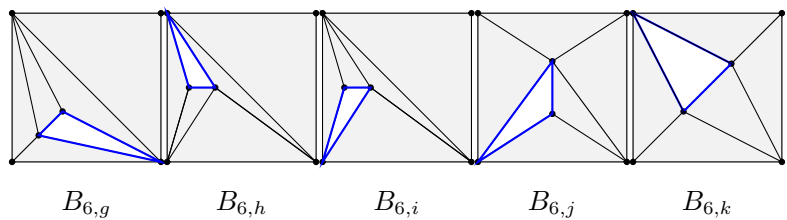

\centering
\bt[scale=0.5]
\coordinate (a) at (45:2.8);
\coordinate (b) at (135:2.8);
\coordinate (c) at (-135:2.8);
\coordinate (d) at (-45:2.8);

\dff(a)--(b)--(c)--(d)--cycle;
\coordinate (e) at (-135:0.9);
\coordinate (f) at (-135:1.8);
\draw[fill,color=white](d)--(e)--(f)--cycle;
\df (a)circle(2pt)--(b)circle(2pt)--(c)circle(2pt)--(d)circle(2pt)--(a);
\df (c)--(f)circle(2pt)--(e)circle(2pt)--(d);
\draw (d)--(b)--(e);
\draw (b)--(f)--(d);
\draw[color=blue,thick](d)--(e)--(f)--cycle;
\draw (0,-2.5)node[below] {$B_{6,g}$};
\et
\bt[scale=0.5]
\coordinate (a) at (45:2.8);
\coordinate (b) at (135:2.8);
\coordinate (c) at (-135:2.8);
\coordinate (d) at (-45:2.8);

\dff(a)--(b)--(c)--(d)--cycle;
\coordinate (e) at (180:0.7);
\coordinate (f) at (180:1.4);
\draw[fill,color=white](b)--(e)--(f)--cycle;
\df (a)circle(2pt)--(b)circle(2pt)--(c)circle(2pt)--(d)circle(2pt)--(a);
\df (c)--(f)circle(2pt)--(e)circle(2pt)--(d);
\draw (d)--(b)--(e)--(c)--(f)--(e)--(d);
\draw (b)--(f);
\draw[color=blue,thick](b)--(e)--(f)--cycle;
\draw (0,-2.5)node[below] {$B_{6,h}$};
\et
\bt[scale=0.5]
\coordinate (a) at (45:2.8);
\coordinate (b) at (135:2.8);
\coordinate (c) at (-135:2.8);
\coordinate (d) at (-45:2.8);

\dff(a)--(b)--(c)--(d)--cycle;
\coordinate (e) at (180:0.7);
\coordinate (f) at (180:1.4);
\draw[fill,color=white](c)--(e)--(f)--cycle;
\df (a)circle(2pt)--(b)circle(2pt)--(c)circle(2pt)--(d)circle(2pt)--(a);
\df (c)--(f)circle(2pt)--(e)circle(2pt)--(d);
\draw (d)--(b)--(e)--(c)--(f)--(e)--(d);
\draw (b)--(f);
\draw[color=blue,thick](c)--(e)--(f)--cycle;
\draw (0,-2.5)node[below] {$B_{6,i}$};
\et
\bt[scale=0.5]
\coordinate (a) at (45:2.8);
\coordinate (b) at (135:2.8);
\coordinate (c) at (-135:2.8);
\coordinate (d) at (-45:2.8);

\dff(a)--(b)--(c)--(d)--cycle;
\coordinate (e) at (90:0.7);
\coordinate (f) at (-90:0.7);
\draw[fill,color=white](c)--(e)--(f)--cycle;
\df (a)circle(2pt)--(b)circle(2pt)--(c)circle(2pt)--(d)circle(2pt)--(a);
\df (c)--(f)circle(2pt)--(e)circle(2pt)--(d);
\draw (d)--(f);
\draw (a)--(e)--(c);
\draw(b)--(e);
\draw[color=blue,thick](c)--(e)--(f)--cycle;
\draw (0,-2.5)node[below] {$B_{6,j}$};
\et
\bt[scale=0.5]
\coordinate (a) at (45:2.8);
\coordinate (b) at (135:2.8);
\coordinate (c) at (-135:2.8);
\coordinate (d) at (-45:2.8);

\dff(a)--(b)--(c)--(d)--cycle;
\coordinate (e) at (45:0.9);
\coordinate (f) at (225:0.9);
\draw[fill,color=white](b)--(e)--(f)--cycle;
\df (a)circle(2pt)--(b)circle(2pt)--(c)circle(2pt)--(d)circle(2pt)--(a);
\df (f)circle(2pt)--(e)circle(2pt)--(a);
\draw[color=blue,thick](b)--(e)--(f)--cycle;
\draw (c)--(f)--(b)--(e)--(d)--(f);
\draw (0,-2.5)node[below] {$B_{6,k}$};
\et
\caption{\label{b6b4}When $B$ has $6$ vertices, boundary forms a $C_4$ and a $C_3$}
\end{figure}
In this case,
\begin{align*}
f_B&\leq 5+{1\over 8}e_B^\partial+{1\over 5}(7-e_B^\partial)={27\over 4}-{1\over 8}e_B^\partial\\
{41\over 72}e_B^\partial+{11\over 18}e_B^{int}&={41\over 72}e_B^\partial+{11\over 18}(11-e_B^\partial)={121\over 18}-{1\over 24}e_B^\partial.
\end{align*}
Therefore, when $e_B^\partial\geq 1$, we have $f_B\leq {41\over 72}e_B^\partial+{11\over 18}e_B^{int}$. It suffices to show that $B$ has at least one edge adjacent to some $(8+)$-face. Notice that for each blue edge, it cannot be adjacent to another triangular face not contained in $B$, otherwise $B$ would be larger. However, if it is adjacent to a face of length between $4$ and $6$ which shares only two vertices with $B$, then we have a $C_7$ without an $(8+)$-face inside. Therefore the exterior face of $B$ adjacent to any blue edge has to use exactly $3$ vertices of $B$. However, it cannot use all three blue edges, otherwise $B$ would be larger. Hence the other blue edge not used by this face has to border an $(8+)$-face.

\textbf{Case 12: }$B$ is $B_{6,l},B_{6,m},B_{6,n},B_{6,o},B_{6,p},B_{6,q}$, or $B_{6,r}$.

\begin{figure}
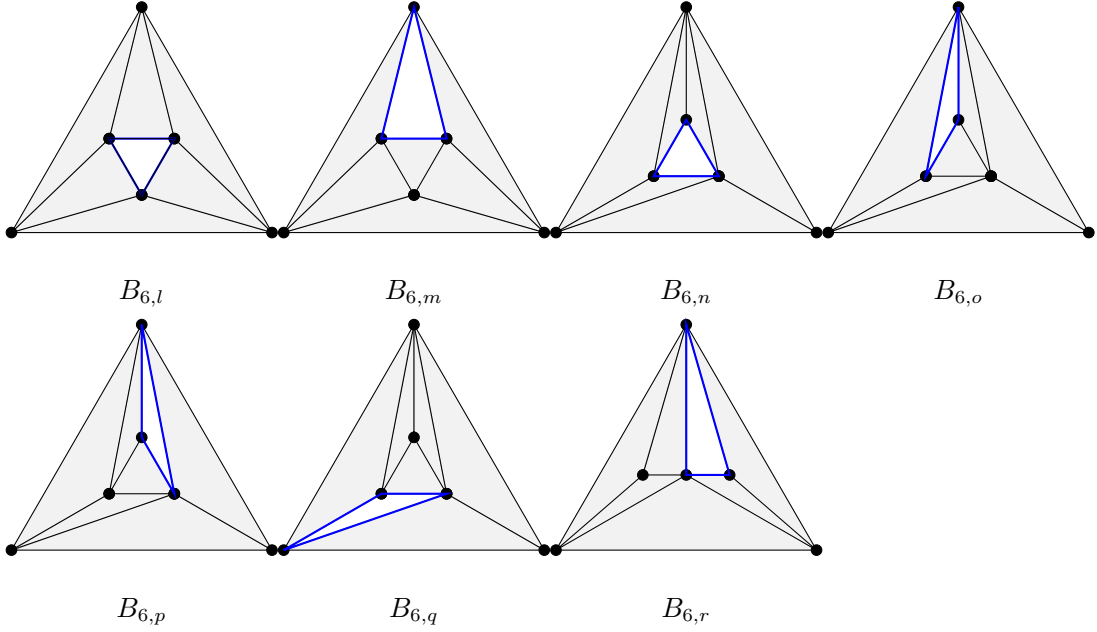

\bt
\dff (90:2)--(-30:2)--(-150:2)--cycle;
\draw[fill,color=white](30:0.5)--(-90:0.5)--(150:0.5)--cycle;
\draw[fill](90:2)circle(2pt);
\draw[fill](-30:2)circle(2pt);
\draw[fill](-150:2)circle(2pt);
\draw (90:2)--(-30:2)--(-150:2)--cycle;
\draw[fill](90:2)--(30:0.5)circle(2pt);
\draw[fill](90:2)--(150:0.5)circle(2pt);
\draw[fill](-30:2)--(30:0.5)circle(2pt);
\draw[fill](-30:2)--(-90:0.5)circle(2pt);
\draw[fill](-150:2)--(-90:0.5)circle(2pt);
\draw[fill](-150:2)--(150:0.5)circle(2pt);
\draw[color=blue,thick](30:0.5)--(-90:0.5)--(150:0.5)--cycle;
\draw (-90:0.5)--(150:0.5)--(30:0.5)--cycle;
\draw (0,-1.5)node[below] {$B_{6,l}$};
\et
\bt
\dff (90:2)--(-30:2)--(-150:2)--cycle;
\draw[fill,color=white](30:0.5)--(90:2)--(150:0.5)--cycle;
\draw[fill](90:2)circle(2pt);
\draw[fill](-30:2)circle(2pt);
\draw[fill](-150:2)circle(2pt);
\draw (90:2)--(-30:2)--(-150:2)--cycle;
\draw[fill](90:2)--(30:0.5)circle(2pt);
\draw[fill](90:2)--(150:0.5)circle(2pt);
\draw[fill](-30:2)--(30:0.5)circle(2pt);
\draw[fill](-30:2)--(-90:0.5)circle(2pt);
\draw[fill](-150:2)--(-90:0.5)circle(2pt);
\draw[fill](-150:2)--(150:0.5)circle(2pt);
\draw (-90:0.5)--(150:0.5)--(30:0.5)--cycle;
\draw[color=blue,thick](30:0.5)--(90:2)--(150:0.5)--cycle;
\draw (0,-1.5)node[below] {$B_{6,m}$};
\et
\bt
\dff (90:2)--(-30:2)--(-150:2)--cycle;
\draw[fill,color=white](90:0.5)--(-30:0.5)--(-150:0.5)--cycle;
\draw[fill](90:2)circle(2pt);
\draw[fill](-30:2)circle(2pt);
\draw[fill](-150:2)circle(2pt);
\draw (90:2)--(-30:2)--(-150:2)--cycle;
\draw[fill](90:2)--(-30:0.5)circle(2pt);
\draw[fill](90:2)--(-150:0.5)circle(2pt);
\draw[fill](90:2)--(90:0.5)circle(2pt);
\draw[fill](-30:2)--(-30:0.5)circle(2pt);
\draw[fill](-150:2)--(-150:0.5)circle(2pt);
\draw[fill](-150:2)--(-30:0.5)circle(2pt);
\draw (90:0.5)--(-150:0.5)--(-30:0.5)--cycle;
\draw[color=blue,thick](90:0.5)--(-30:0.5)--(-150:0.5)--cycle;
\draw (0,-1.5)node[below] {$B_{6,n}$};
\et
\bt
\dff (90:2)--(-30:2)--(-150:2)--cycle;
\draw[fill,color=white](90:0.5)--(90:2)--(-150:0.5)--cycle;
\draw[fill](90:2)circle(2pt);
\draw[fill](-30:2)circle(2pt);
\draw[fill](-150:2)circle(2pt);
\draw (90:2)--(-30:2)--(-150:2)--cycle;
\draw[fill](90:2)--(-30:0.5)circle(2pt);
\draw[fill](90:2)--(-150:0.5)circle(2pt);
\draw[fill](90:2)--(90:0.5)circle(2pt);
\draw[fill](-30:2)--(-30:0.5)circle(2pt);
\draw[fill](-150:2)--(-150:0.5)circle(2pt);
\draw[fill](-150:2)--(-30:0.5)circle(2pt);
\draw (90:0.5)--(-150:0.5)--(-30:0.5)--cycle;
\draw[color=blue,thick](90:0.5)--(90:2)--(-150:0.5)--cycle;
\draw (0,-1.5)node[below] {$B_{6,o}$};
\et
\\
\bt
\dff (90:2)--(-30:2)--(-150:2)--cycle;
\draw[fill,color=white](90:0.5)--(90:2)--(-30:0.5)--cycle;
\draw[fill](90:2)circle(2pt);
\draw[fill](-30:2)circle(2pt);
\draw[fill](-150:2)circle(2pt);
\draw (90:2)--(-30:2)--(-150:2)--cycle;
\draw[fill](90:2)--(-30:0.5)circle(2pt);
\draw[fill](90:2)--(-150:0.5)circle(2pt);
\draw[fill](90:2)--(90:0.5)circle(2pt);
\draw[fill](-30:2)--(-30:0.5)circle(2pt);
\draw[fill](-150:2)--(-150:0.5)circle(2pt);
\draw[fill](-150:2)--(-30:0.5)circle(2pt);
\draw (90:0.5)--(-150:0.5)--(-30:0.5)--cycle;
\draw[color=blue,thick](90:0.5)--(90:2)--(-30:0.5)--cycle;
\draw (0,-1.5)node[below] {$B_{6,p}$};
\et
\bt
\dff (90:2)--(-30:2)--(-150:2)--cycle;
\draw[fill,color=white](-30:0.5)--(-150:2)--(-150:0.5)--cycle;
\draw[fill](90:2)circle(2pt);
\draw[fill](-30:2)circle(2pt);
\draw[fill](-150:2)circle(2pt);
\draw (90:2)--(-30:2)--(-150:2)--cycle;
\draw[fill](90:2)--(-30:0.5)circle(2pt);
\draw[fill](90:2)--(-150:0.5)circle(2pt);
\draw[fill](90:2)--(90:0.5)circle(2pt);
\draw[fill](-30:2)--(-30:0.5)circle(2pt);
\draw[fill](-150:2)--(-150:0.5)circle(2pt);
\draw[fill](-150:2)--(-30:0.5)circle(2pt);
\draw (90:0.5)--(-150:0.5)--(-30:0.5)--cycle;
\draw[color=blue,thick](-30:0.5)--(-150:2)--(-150:0.5)--cycle;
\draw (0,-1.5)node[below] {$B_{6,q}$};
\et
\bt
\coordinate (a) at (0.57735,0);
\coordinate (b) at (-0.57735,0);
\dff (90:2)--(-30:2)--(-150:2)--cycle;
\draw[fill,color=white](90:2)--(a)--(0,0)--cycle;
\draw[fill](90:2)circle(2pt);
\draw[fill](-30:2)circle(2pt);
\draw[fill](-150:2)circle(2pt);
\draw (90:2)--(-30:2)--(-150:2)--cycle;
\draw[fill](90:2)--(0,0)circle(2pt);
\draw[fill](90:2)--(a)circle(2pt);
\draw[fill](90:2)--(b)circle(2pt);
\draw(0,0)--(-150:2)--(b)--(0,0)--(a)--(-30:2)--(0,0);
\draw[color=blue,thick](90:2)--(a)--(0,0)--cycle;
\draw (0,-1.5)node[below] {$B_{6,r}$};
\et
\caption{\label{b6b3}When $B$ has $6$ vertices, boundary forms two $C_3$'s}
\end{figure}
In this case,
\begin{align*}
f_B&\leq 5+{1\over 8}e_B^\partial+{1\over 5}(7-e_B^\partial)={27\over 4}-{1\over 8}e_B^\partial\\
{41\over 72}e_B^\partial+{11\over 18}e_B^{int}&={41\over 72}e_B^\partial+{11\over 18}(11-e_B^\partial)={121\over 18}-{1\over 24}e_B^\partial.
\end{align*}
Thus $f_B\leq {41\over 72}e_B^\partial+{11\over 18}e_B^{int}$ when $e_B^\partial\geq 2$. It can be easily verified that all three blue edges are adjacent to some $(8+)$-faces, so $e_B^\partial\geq 3$.

\end{proof}
\begin{lemma}\label{smallblocksonly}
Let $G'$ be a $2$-connected, plane graph on $n$ ($n\geq 8$) vertices with $\delta(G')\geq 3$ and with any adjacent vertices having total degree at least $7$. Let $G$ be the plane graph obtained from $G'$ by doing the replacement described in Lemma \ref{replacement}. Assume $G$ doesn't have $(8+)$-faces. Then, we have the following inequality:
    \begin{equation*}
        f(G) \leq {9\over 16}e(G).
    \end{equation*}

\end{lemma}
Note that this is a much stronger result than Lemma \ref{smallblockslemma}.
\begin{proof}
Just like the proof of Lemma \ref{smallblockslemma}, we will have to run through all the different small triangular block cases, and show that the face and edge contributions satisfies the inequality. What makes this proof easier will be the fact that we can ``expand" out block to include its exterior faces, to show that either exterior faces of the block are large enough such that the inequality is satisfied; or that we eventually get to a point where the (expanded) block can only border $(8+)$-faces. Then it follows that either all of these $(8+)$-faces are the same face, then $G$ has a cut vertex, or there are $2$ or more $(8+)$-faces, both of which are contradictions. \\

We will apply this technique to the most difficult small triangular block, the $B_{4,b}$, and we note that for all the other blocks, similar calculations and reasoning can be used. We also showcase the calculation of the trivial block ($B_2$), and a typical small triangular block with an ``inner boundary" ($B_{6,i}$). Denote $B$ as the triangular block we will study. Similar to our proof of Lemma \ref{smallblockslemma}, we use blue area to represent a shape that is not a face and contains vertices inside. If we draw a $C_7$ with red edges, it means each of the edge could not be part of a $T_3$. Thus these red edges are not added through the process described in Lemma \ref{replacement} and should not form a $C_7$.

\textbf{Case 1: }$B$ is $B_2$.
\begin{align*}
f_B&\leq {1\over 4}+{1\over 4}={1\over 2}\\
{9\over 16}e_B&={9\over 16}\\
f_B&\leq {9\over 16}.
\end{align*}

\textbf{Case 2: }$B$ is $B_{6,i}$.

As explained in Case 12 of the proof of Lemma \ref{smallblockslemma}, $e_B^\partial\geq 1$ so there is an $(8+)$-face.

\textbf{Case 3: }$B$ is $B_{4,b}$.

Similar to Case 4 of the proof of Lemma \ref{smallblockslemma}, the only possible drawing (up to symmetry) where two boundary edges of $B$ bordering some face of length $4$ lead to $e_B^\partial\geq 1$.

If exactly one edge of $B$ is adjacent to some face of length $4$, the we claim that all other three edges of that $4$-face is trivial blocks so this face can be uniquely assigned to $B$. The trick to verify this (and to enumerate similar cases) is that, if we don't allow $C_7$ without an $(8+)$-face inside, then some vertices have to be ``used twice" to eliminate this possibility. For example, if any of the other three edges of that $4$-face is not trivial block, as shown in Figure \ref{samever}, the red vertices at the left have to be the same, so the graph looks like the right side.  In this case, the red edge has to be incident to an $(8+)$-face. Otherwise, either $B$ should be larger, or we have a $C_7$ where each edges doesn't belong to some $T_3$.

Using the trick about a vertex ``used twice", we can enumerate all possible drawings when an edge of the $4$-face in $B$ borders another $4$-face as, as shown in Figure \ref{2f4}. In all these drawings, either we find a $C_7$ where each edge doesn't belong to some $T_3$, or we find one of the red edges has to be border an $(8+)$-face(represented by red edges that does not form a $C_7$). Specifically, in $(1)$, if none of the red edges are adjacent to $(8+)$-face, then in order to avoid forming a $C_7$ in $B$, they must be in the same $6$-face, which would form a $C_7$ with the white $4$-face.

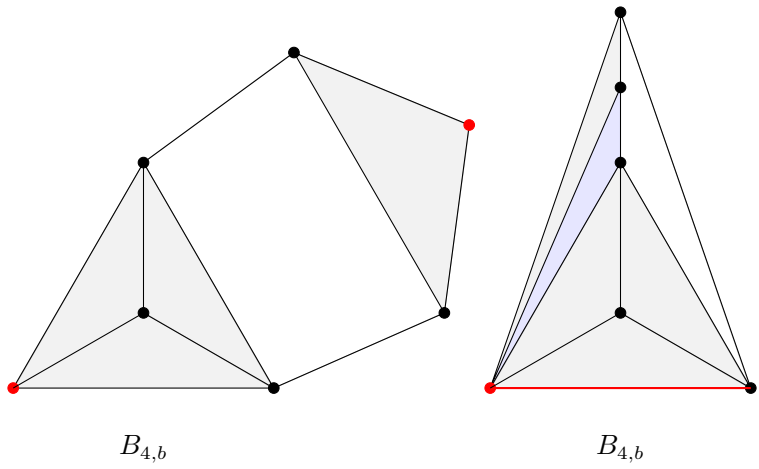
\begin{figure}
\centering
\begin{tikzpicture}
\draw[fill,color=gray!10] (-30:2)--(90:2)--(210:2)--cycle;
\dff (30:5)--(60:4)--(0:4)--cycle;
\draw[fill] (-30:2)circle(2pt)--(90:2)circle(2pt)--(210:2);
\draw (-30:2)--(210:2);
\draw (0,-1.5) node[below] {$B_{4,b}$};
\draw[fill] (0,0)circle(2pt)--(-30:2);
\df (90:2)--(60:4)circle(2pt)--(30:5);
\df (30:5)--(0:4)circle(2pt)--(-30:2);
\draw(60:4)--(0:4);
\draw[fill,color=red](30:5)circle(2pt);
\draw[fill,color=red](210:2)circle(2pt);
\draw (0,0)--(90:2);
\draw (0,0)--(210:2);
\end{tikzpicture}
\begin{tikzpicture}
\draw[fill,color=gray!10] (-30:2)--(90:2)--(210:2)--cycle;
\dff (90:3)--(90:4)--(210:2)--cycle;
\draw[fill,color=blue!10](90:3)--(90:2)--(210:2)--cycle;
\draw[fill] (-30:2)circle(2pt)--(90:2)circle(2pt)--(210:2);
\draw (-30:2)--(210:2);
\draw (0,-1.5) node[below] {$B_{4,b}$};
\draw[fill] (0,0)circle(2pt)--(-30:2);
\df (90:2)--(90:3)circle(2pt)--(90:4)circle(2pt)--(-30:2);
\draw (90:3)--(210:2)--(90:4);
\dfr (-30:2)--(210:2);

\draw[fill,color=red](210:2)circle(2pt);
\draw (0,0)--(90:2);
\draw (0,0)--(210:2);
\end{tikzpicture}
\caption{\label{samever}The red vertices are the same vertex}
\end{figure}

Now when we assign the face of length $4$ to $B$, we can assume that each of its boundary edges has to border a face with length at least $5$. Then,
\begin{align*}
f_B&\leq 4+{5\over 5}=5\\
{9\over 16}e_B&={81\over 16}>5\\
f_B&\leq {9\over 16}e_B.
\end{align*}

\begin{figure}
\centering
\begin{tikzpicture}[scale=0.8]
\draw[fill,color=gray!10] (-30:2)--(90:2)--(210:2)--cycle;

\dff(90:2)--(30:3)--(30:2)--(-30:2)--cycle;

\dfb(30:3)--(30:2)--(-30:2)--cycle;
\dfr(30:3)--(30:2)--(-30:2);
\df(90:2)--(30:3)circle(2pt)--(-30:2);
\df(-30:2)--(30:4.5)circle(2pt);
\draw(30:4.5)--(90:2);
\df(30:2)circle(2pt);

\draw[fill] (-30:2)circle(2pt)--(90:2)circle(2pt)--(210:2)circle(2pt)--(-30:2);
\draw (0,-1.5) node[below] {$(1)$};
\draw[fill] (0,0)circle(2pt)--(-30:2);

\draw (0,0)--(90:2);
\draw (0,0)--(210:2);
\end{tikzpicture}
\begin{tikzpicture}[scale=0.8]
\draw[fill,color=gray!10] (-30:2)--(90:2)--(210:2)--cycle;
\dff(90:2)--(15:2.5)--(30:4.5)--(-30:2)--cycle;
\dfb(15:2.5)--(40:2.8)--(30:4.5)--cycle;
\draw[fill] (-30:2)circle(2pt)--(90:2)circle(2pt)--(210:2)circle(2pt)--(-30:2);
\draw (0,-1.5) node[below] {$(2)$};
\draw[fill] (0,0)circle(2pt)--(-30:2);
\draw (0,0)--(90:2);
\draw (0,0)--(210:2);
\dfr (90:2)--(210:2)--(0,0)--(-30:2)--(30:4.5)--(40:2.8)--(15:2.5)--cycle;
\df(90:2)--(30:4.5)circle(2pt);
\df(15:2.5)circle(2pt)--(30:4.5);
\df(40:2.8)circle(2pt);
\end{tikzpicture}
\begin{tikzpicture}[scale=0.8]
\draw[fill,color=gray!10] (-30:2)--(90:2)--(210:2)--cycle;
\dff(90:2)--(30:4.5)--(40:2)--(-30:2)--cycle;
\dfb(40:2)--(15:2.2)--(-30:2)--cycle;
\draw[fill] (-30:2)circle(2pt)--(90:2)circle(2pt)--(210:2)circle(2pt)--(-30:2);
\draw (0,-1.5) node[below] {$(3)$};
\draw[fill] (0,0)circle(2pt)--(-30:2);
\draw (0,0)--(90:2);
\draw (0,0)--(210:2);
\dfr(90:2)--(0,0)--(210:2)--(-30:2)--(15:2.2)--(40:2)--(30:4.5)--cycle;
\df(40:2)circle(2pt)--(-30:2);
\df(30:4.5)circle(2pt)--(-30:2);
\df(15:2.2)circle(2pt);
\end{tikzpicture}

\begin{tikzpicture}
\draw[fill,color=gray!10] (-30:2)--(90:2)--(210:2)--cycle;
\dff (-30:2)--(45:4)--(135:4)--(90:2)--cycle;
\dfb(90:2)--(150:2)--(210:2)--cycle;

\draw[fill] (-30:2)circle(2pt)--(90:2)circle(2pt)--(210:2)circle(2pt)--(-30:2);
\dfr(90:2)--(150:2)--(210:2)--(0,0)--(-30:2)--(45:4)--(135:4)--cycle;
\df(45:4)circle(2pt);
\df(135:4)circle(2pt)--(210:2);
\df(150:2)circle(2pt);
\draw (0,-1.5) node[below] {$(4)$};
\draw[fill] (0,0)circle(2pt);
\draw (0,0)--(90:2);

\end{tikzpicture}
\begin{tikzpicture}
\draw[fill,color=gray!10] (-30:2)--(90:1)--(210:2)--cycle;
\dff (-30:2)--(45:4)--(150:2)--(90:1)--cycle;
\dfb(90:1)--(150:2)--(210:2)--cycle;

\draw[fill] (-30:2)circle(2pt)--(90:1)circle(2pt)--(210:2)circle(2pt)--(-30:2);
\dfr(90:1)--(150:2)--(45:4)--(135:4)--(210:2)--(-30:2)--(0:0)--cycle;
\df(45:4)circle(2pt);
\df(135:4)circle(2pt);
\draw(150:2)--(210:2);
\df(150:2)circle(2pt);
\draw (0,-1.5) node[below] {$(5)$};
\draw[fill] (0,0)circle(2pt);
\draw(0,0)--(210:2);
\draw (0,0)--(90:1);
\draw(45:4)--(-30:2);

\end{tikzpicture}

\caption{\label{2f4}An edge of $4$-face borders another $4$-face}
\end{figure}
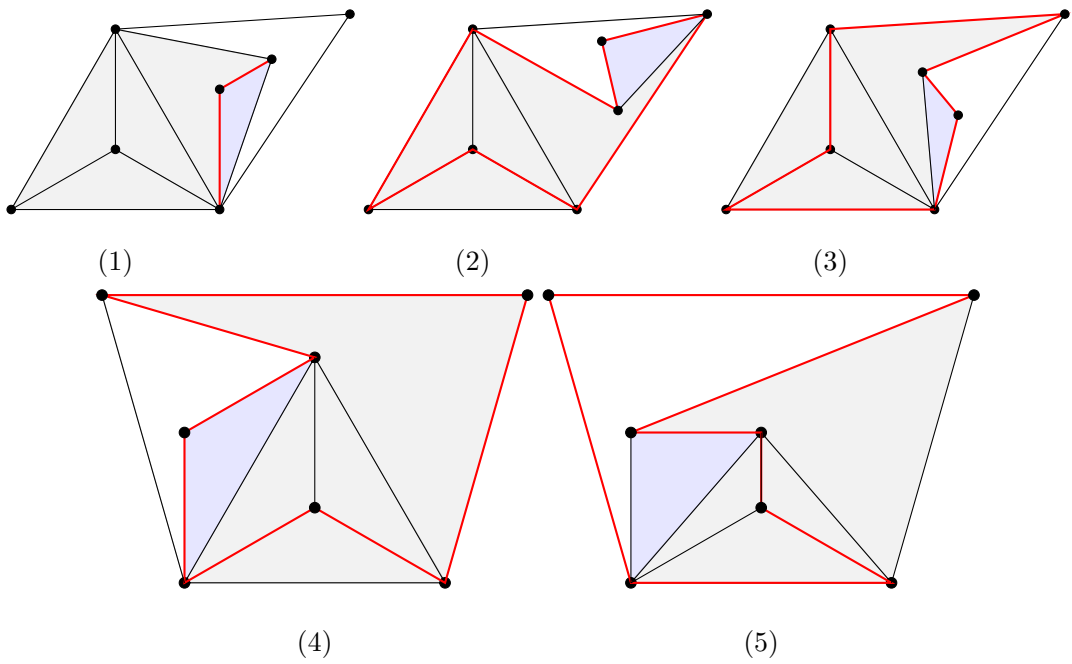
\end{proof}

\section{Proof of Theorem \ref{pt}}
\begin{proof}
Let $G'$ be a $2$-connected, $C_7$-free plane graph on $n$ $(n\geq 8)$ vertices with $\delta(G')\geq 3$ and any adjacent vertices having degree sum at least $7$. We will prove $e(G')\leq {18\over 7}n-{48\over 7}$ in the follow way: by Lemma \ref{replacement}, we can use $T_3$ to replace large triangular blocks, edges incident to $2$ $(8+)$-faces, and small block sets with less than $7$ vertices to get graph $G$ and it suffices to prove $e(G)\leq {18\over 7}v(G)-{48\over 7}$.
\\

Denote $a_i$ as the number of faces in $G$ of length $i$, where $i \geq 8$. For simplicity of notations, denote $$k_1 = \sum_{B \in S(G)} e_B^\partial,k_2=\sum_{B\in S(G)}e_B^{int}.$$ Now treating every $T_3$ and small block set as a single face, we obtain the dual graph denoted as $G^*$. By construction of $G$, no two $(8+)$-faces border each other. By Definition \ref{sbs} and Lemma \ref{adjacencylemma}, small block sets and $T_3$ only border $(8+)$-faces. Thus $G^*$ is a bipartite graph between set of vertices from $T_3$'s in $G$ and set of vertices from small block sets and $(8+)$-faces in $G$. Let $L$ be the number of $T_3$'s in $G$, we have \[ L = \frac{1}{3} ((\sum_{i = 8}^{\infty} ia_i) - k_1) .\] Denote $\hat{G^*}$ as $G^*$ with all double edges removed, and with all small block set vertices and incident edges removed. Clearly $\hat{G^*}$ must also be bipartite and that \[ v(\hat{G^*}) = (\sum_{i = 8}^{\infty} a_i) + L  = (\sum_{i = 8}^{\infty} a_i) + \frac{1}{3} ((\sum_{i = 8}^{\infty} ia_i) - k_1) .\] By property $(9)$ of Lemma \ref{replacement}, if $F$ is a face of $G$ with length $i$ and $i>8$, then $F$ is adjacent to at most $i-8$ $T_3$'s. Thus, if $v_F$ is the vertex in $\hat{G^*}$ coming from $F$, then $v_F$ has at most $i-8$ pairs of double edges into vertices from $T_3$'s of $G$. Hence the construction of $\hat{G^*}$ implies \[
e(\hat{G}^*) \geq (\sum_{i=8}^{\infty} a_i (i - (i-8))- k_1) = (8\sum_{i=8}^{\infty} a_i) - k_1.
\] \\
\textbf{Case 1:} $v(\hat{G^*})\geq 2$.
\\

If $e(\hat{G^*})=0$, clearly $$v(\hat{G}^*) \geq \frac{1}{2}e(\hat{G}^*) + 2.$$If $e(\hat{G^*})\geq 1$, denote $S_0$ as the set of isolated vertices of $\hat{G^*}$ and $S_1$ be the set of connected compoents of $\hat{G^*}$ each with at least $1$ edge, so $|S_1|\geq 1$. By Lemma \ref{replacement}, each $T_3$ must be adjacent to two disjoint $(8+)$-faces, so the vertex in $\hat{G^*}$ coming from this $T_3$ must have degree at least $2$. Therefore, for any $G_i\in S_1$, since $\hat{G^*}$ is bipartite, $G_i$ contains at least $1$ vertex coming from $T_3$ and $e(G_i)\geq 2$. Since $G_i$ is bipartite, it follows that
\begin{align*}
f(G_i)&\leq {1\over 2}e(G_i)\\
2+e(G_i)-v(G_i)&\leq {1\over 2}e(G_i)\\
v(G_i)&\geq {1\over 2}e(G_i)+2.
\end{align*}
Therefore,
\begin{align*}
v(\hat{G^*})&=\sum_{G_i\in S_1}v(\hat{G^*})\\
&\geq \sum_{G_i\in S_1}({1\over 2}e(G_i)+2)\\
&=\sum_{G_i\in S_1}{1\over 2}e(G_i)+2|S_1|\\
&\geq{1\over 2}e(\hat{G^*})+2.
\end{align*}
Hence we conclude that if $v(\hat{G^*})\geq 2$, then
\begin{align*}
v(\hat{G^*})&\geq {1\over 2}e(\hat{G^*})+2\\
(\sum_{i=8}^{\infty} a_i) + \frac{1}{3}(\sum_{i=8}^{\infty} ia_i - k_1) & \geq \frac{1}{2}(\sum_{i=8}^{\infty} 8a_i) - \frac{1}{2}k_1 + 2 \\
(\sum_{i=8}^{\infty} ia_i) - k_1 & \geq 9(\sum_{i=8}^{\infty} a_i) - \frac{3}{2}k_1 + 6
\end{align*}
\begin{equation}\label{eq1}
\sum_{i=8}^{\infty} a_i  \leq  \frac{1}{9}(\sum_{i=8}^{\infty} ia_i) + \frac{1}{18}k_1 - \frac{2}{3}.
\end{equation}

Now we go back to the primal graph $G$. Denote $f_k$ as the number of faces in the interiors of all small block sets. Then we see that
\begin{equation} \label{eq2}
\begin{split}
e(G) & = \text{\# edges in $T_3$ blocks + \# edges in all small block sets} \\
& = \frac{12}{3}((\sum_{i=8}^{\infty} ia_i) - k_1) + (k_1+k_2)\\
&=4\sum_{i=8}^\infty ia_i-3k_1+k_2
\end{split}
\end{equation}

\begin{equation} \label{eq3}
\begin{split}
f(G) & = \text{\# faces in the interior of $T_3$ blocks } \\
& + \text{  \# faces in the interior of small block sets } \\
& + \text{  \# of 8+faces} \\
& = \frac{7}{3}((\sum_{i=8}^{\infty} ia_i) - k_1) + (\sum_{i=8}^{\infty} a_i) + f_k
\end{split}
\end{equation}

Note that $\sum_{B\in S(G)} f_B = f_k + \text{face contribution on the boundary} = f_k + \frac{1}{8}k_1$, so lemma \ref{smallblockslemma} implies that:
\begin{equation} \label{eq4}
f_k \leq \frac{41}{72}\sum_{B\in S(G)} e_B^{\partial} + \frac{11}{18}\sum_{B\in S(G)} e_B^{\text{int}} - \frac{1}{8}k={4\over 9}k_1+{11\over 18}k_2.
\end{equation}

Now we combine Equations \ref{eq1}, \ref{eq2}, \ref{eq3}, \ref{eq4} together using Euler's formula for $G$:
\begin{align*}
v(G) & = 2 + e(G) - f(G) \\
& = 2+4\sum_{i=8}^\infty ia_i-3k_1+k_2-{7\over 3}\sum_{i=8}^\infty ia_i+{7\over 3}k_1-\sum_{i=8}^\infty a_i-f_k\\
&={5\over 3}\sum_{i=8}^\infty ia_i-\sum_{i=8}^\infty a_i-{2\over 3}k_1+k_2+2-f_k\\
&\geq {5\over 3}\sum_{i=8}^\infty ia_i-{2\over 3}k_1+k_2+2-{4\over 9}k_2-{11\over 18}k_2-{1\over 9}\sum_{i=8}^\infty ia_i-{1\over 18}k_2+{2\over 3}\\
&={14\over 9}\sum_{i=8}^\infty ia_i-{7\over 6}k_1+{7\over 18}k_2+{8\over 3}\\
{e(G)\over v(G)-{8\over 3}}&\leq {4\sum_{i=8}^\infty ia_i-3k_1+k_2\over {14\over 9}\sum_{i=8}^\infty ia_i-{7\over 6}k_1+{7\over 18}k_2}={18\over 7}\\
\Rightarrow e(G)&\leq {18\over 7}v(G)-{48\over 7}.
\end{align*}
\\
\textbf{Case 2:} $v(\hat{G^*})=1$.
\\
Observe that this vertex has to be a vertex coming from a $(8+)$-face in $G$ and the size of this face is just $k_1$. Similar to Case 1, we have
\begin{align*}
f(G)&\leq {4\over 9}k_1+{11\over 18}k_2\\
e(G)&=k_1+k_2\\
v(G)&=2+e(G)-f(G)\\
&\geq 2+k_1+k_2-{4\over 9}k_1-{11\over 18}k_2\\
&=2+{5\over 9}k_1+{7\over 18}k_2\\
{e(G)\over v(G)-{8\over 3}}&\leq {k_1+k_2\over 2+{5\over 9}l+{7\over 18}k_2-{8\over 3}}\\
&={k_1+k_2\over {5\over 9}k_1-{2\over 3}+{7\over 18}k_2}.
\end{align*}
Since ${k_2\over {7\over 18}k_2}={18\over 7}$, it suffices to prove ${k_1\over {5\over 9}k_1-{2\over 3}}\leq {18\over 7}$. Notice that $k_1$ is the size of a $(8+)$-face, so $k_1\geq 8$. It follows that $7k_1\leq 10k_1-12,{k_1\over {5\over 9}k_1-{2\over 3}}\leq {18\over 7}$.
\\

\textbf{Case 3:} $v(\hat{G^*})=0$.
\\

If $\hat{G}^*$ were to be empty, then it must mean that $G$ consists only of small blocks and no $(8+)$-faces. Then by lemma \ref{smallblocksonly}, we get:
\begin{equation}
\begin{split}
f(G) & \leq {9\over 16}e(G) \\
2 + e(G) - v(G) & \leq {9\over 16}e(G) \\
e(G) & \leq \frac{16}{7}v(G) - \frac{32}{7} \\
& = \frac{18}{7}v(G) - (\frac{32}{7} + \frac{2}{7}v(G)) \\
& \leq \frac{18}{7}v(G) - (\frac{32}{7} + \frac{16}{7}) \text{   (as $v(G) \geq 8$)} \\
& = \frac{18}{7}v(G) - \frac{48}{7}.
\end{split}
\end{equation}
\end{proof}

\section{Proof of Theorem \ref{mt}}
\begin{proof}
Let $G$ be a $C_7$-free plane graph on $n$ vertices. For simplicity of notations, in this section we denote $e(G)$ as $e$. We will show that either $e(G)\leq {18\over 7}n-{48\over 7}$, which is equivalent with $18n-7e\geq 48$, or $n<60$. We will do the following operations:

$1)$ Delete $x$ s.t. $deg(x)\leq 2$;

$2)$ Delete $x,y$ s.t. $(xy)\in E(G)$ and $deg(x)+deg(y)\leq 6$.

When we no longer can do this operation, we get an induced plane subgraph $G'$ with $\delta(G')\geq 3$ and each adjacent vertices having degree sum at least $7$. Denote $|E(G')|=e',|V(G')|=v'$.
Each time we do operation $1)$ or $2)$, denote the number of vertices deleted as $v_d$ and the number of edges deleted as $e_d$. Observe that ${18v_d-7e_d\over v_d}\geq 0.5$. Thus we get $$18n-7e\geq 18n'-7e'+0.5(n-n').$$

In line with usual graph theoretic terminology, we call a maximal $2$-connected subgraph a \textbf{block}. Let $b$ be the total number of blocks of $G'$. Specifically, let $b_2,b_3,b_4,b_5,b_6$, and $b_7$ denote the number of blocks of size $2,3,4,5,6$, and $7$, respectively. Let $b_8$ denote the number of blocks of size at least $8$. Then we have $b=\sum_{i=2}^8b_i$.

Let's find a lowerbound of $18n-7e-18$ for blocks of size $n$. Since $n$ is fixed, we need to plug in the largest possible $e$, which ideally would be the number of edges in a triangulated plane graph.

When $n\geq 8$, by Theorem \ref{pt} we have $18n-7e-18\geq 48-18=30$.

When $n=7$, since any triangulated plane graph contains a $C_7$, we have $18n-7e-18\geq 18\cdot7-7\cdot14-18=10.$

When $n= 6$, as in $T_3$ we have $18n-7e-18\geq 18\cdot6-7\cdot12-18=6$.

When $n= 5$, as in $B_{5,d}$ we have $18n-7e-18\geq 18\cdot5-7\cdot9-18=9$.

When $n=4$, as in $B_{4,b}$ we have $18n-7e-18\geq 18\cdot4-7\cdot6-18=12$.

When $n=3$, as in $B_{3}$ we have $18n-7e-18\geq 18\cdot3-7\cdot3-18=15$.

When $n=2$, as in $B_{2}$ we have $18n-7e-18\geq 18\cdot2-7\cdot1-18=11$.
\\
\\

Denote the vertex number of the $i$th block as $n_i$ and edge number as $e_i$. Using the lowerbounds presented above, when $n'\neq n$, we have
\begin{align*}
18n'-7e'&\geq18(\sum_{i=1}^{b}n_i-(b-1))-7\sum_{i=1}^{b}e_i\\
&=\sum_{i=1}^b(18n_i-7e_i-18)+18\\
&\geq 30b_8+10b_7+6b_6+9b_5+12b_4+15b_3+11b_2+18\\
18n-7e&\geq 30b_8+10b_7+6b_6+9b_5+12b_4+15b_3+11b_2+18+{1\over 2}(n-n').
\end{align*}

If $b_8\geq 1$, then righthand side is at least $48$, as desired. 

So, let us assume that $b_8=0$. It follows that 
\begin{align*}
b&=\sum_{i=2}^7b_i\\
n&\leq \sum_{i=2}^7ib_i+(n-n')\\
18n-7e&\geq 30b_8+10b_7+6b_6+9b_5+12b_4+15b_3+11b_2+18+{1\over 2}(n-n')\\
&= 10b_7+6b_6+9b_5+12b_4+15b_3+11b_2+18+{1\over 2}(n-n')\\
&\geq 18+\sum_{i=2}^7{i\over 2}b_i+{1\over 2}(n-n')\\
&\geq 18+{1\over 2}n.
\end{align*}
This is larger or equal to $48$ if $n\geq 60$.

Finally, if $n=n'$ and $n\geq 2$, since the last two vertices we delete has at most one edge between them, we know $$18n-7e\geq 18\cdot 2-7\cdot 1+{1\over 2}(n-2)=29+{1\over 2}(n-2)$$which is larger or equal to $48$ when $n\geq 40$. Since $40<60$, we complete the proof of Theorem \ref{mt}.

{\bf Added in proof:} 
Before finishing our paper we just learned that Ruilin Shi, Zach Walsh, Xingxing Yu
determined the planar Tur\'an number of the 7-cycle too. (See
arXiv:2306.13594) That proof has some similarities and differences with our proof, so the proofs are independent.

\end{proof}
\bibliographystyle{abbrv}
\bibliography{c7}

\end{document}